\documentclass[oneside,english,12pt]{amsart}
\usepackage{lmodern}
\usepackage[T1]{fontenc}       			 
\usepackage[latin9]{inputenc}  			 
\usepackage[english]{babel}			 
\usepackage[top=3cm, bottom=3cm, left=3cm, right=3cm, heightrounded, marginparwidth=2.5cm, marginparsep=6mm]{geometry}

\usepackage{setspace} 
\usepackage{marginnote}
\usepackage{verbatim}
\usepackage{mathrsfs}	
\usepackage{amstext}
\usepackage{amsthm}
\usepackage{amssymb}
\usepackage{amsopn} 
\usepackage{bbold}
\usepackage{stix}
\usepackage{enumitem}				
\usepackage[all]{xy}    				
\usepackage{mathdots}   				
\usepackage{mathtools}
\usepackage{aliascnt}   				
\usepackage{esint}					
\usepackage{tikz}  \usetikzlibrary{matrix}	
\usepackage[backref=page]{hyperref}
\usepackage{mathdots}
\hypersetup{
    colorlinks,
    linkcolor={red!50!black},
    citecolor={green!50!black},
    urlcolor={blue!80!black},
    linktocpage
}
\usepackage{soul}  
\usepackage{longtable} 
\usepackage{lipsum}
\usepackage{footnote}
\makesavenoteenv{tabular}

\usepackage{multirow}
\usepackage{multicol} 	
\allowdisplaybreaks

\renewcommand*{\backref}[1]{}
\renewcommand*{\backrefalt}[4]{%
    \ifcase #1 (Not cited.)%
    \or        (Cited on page~#2.)%
    \else      (Cited on pages~#2.)%
    \fi}

\numberwithin{equation}{section}
\numberwithin{figure}{section}

\theoremstyle{theorem}
  \newtheorem*{cor*}{Corollary}
  \newtheorem*{conj*}{Isomorphism Conjecture}
  \newtheorem*{thm*}{Theorem}
  \newtheorem*{lem*}{Lemma}
  \newtheorem*{klem*}{Key Lemma}
  \newtheorem*{klemb*}{Key Lemma \emph{bis}}
  \newtheorem*{claim*}{Claim}
  \newtheorem*{mclaim*}{Main Claim}
  \newtheorem*{rep_thm1}{Theorem \ref{main_thm} \emph{bis}}
  \newtheorem*{rep_thm2}{Theorem \ref{new_stability} \emph{bis}}
  \newtheorem*{rep_thm3}{Theorem \ref{restriction_inj} \emph{bis}}
  
  \newtheorem{thmx}{Theorem}

  \newaliascnt{propx}{thmx}
  \newtheorem{propx}[propx]{Proposition}
  \aliascntresetthe{propx}

  \newtheorem{thm}{Theorem}[section]

  \newaliascnt{conj}{thm}
  
  \aliascntresetthe{conj}

  \newaliascnt{lem}{thm}
  \newtheorem{lem}[lem]{Lemma}
  \aliascntresetthe{lem}

  \newaliascnt{klem}{thm}
  
  \aliascntresetthe{klem}

  \newaliascnt{cor}{thm}
  \newtheorem{cor}[cor]{Corollary}
  \aliascntresetthe{cor}

  \newaliascnt{prop}{thm}  
  \newtheorem{prop}[prop]{Proposition}
  \aliascntresetthe{prop}

\theoremstyle{definition}
  \newaliascnt{defn}{thm}
  \newtheorem{defn}[defn]{Definition}
  \aliascntresetthe{defn}

  \newcommand{\continuation}{??}

    \newaliascnt{nnot}{thm}
  
  \aliascntresetthe{nnot}  
  
  \newaliascnt{exmpl}{thm}
  \newtheorem{exmpl}[exmpl]{Example}
  \aliascntresetthe{exmpl}

\theoremstyle{remark}
  \newaliascnt{rem}{thm}
  \newtheorem{rem}[rem]{Remark}
  \aliascntresetthe{rem}

  \theoremstyle{remark}
  \newaliascnt{con}{thm}
  
  \aliascntresetthe{con}

\newcommand{\bbC}{\mathbb{C}}

\newcommand{\bbF}{\mathbb{F}}

\newcommand{\bbH}{\mathbb{H}}

\newcommand{\bbN}{\mathbb{N}}

\newcommand{\bbP}{\mathbb{P}}

\newcommand{\bbR}{\mathbb{R}}

\newcommand{\bbZ}{\mathbb{Z}}

\newcommand{\clB}{\mathcal{B}}

\newcommand{\clD}{\mathcal{D}}
\newcommand{\clE}{\mathcal{E}}
\newcommand{\clF}{\mathcal{F}}

\newcommand{\clH}{\mathcal{H}}
\newcommand{\clI}{\mathcal{I}}

\newcommand{\clP}{\mathcal{P}}
\newcommand{\clQ}{\mathcal{Q}}

\newcommand{\frS}{\mathfrak{S}}

\newcommand{\scB}{\mathscr{B}}

\newcommand{\scF}{\mathscr{F}}

\newcommand{\scL}{\mathscr{L}}

\DeclareMathOperator{\Aut}{Aut}

\DeclareMathOperator{\CR}{cr}
\DeclareMathOperator{\dd}{d}
\DeclareMathOperator{\EE}{E}

\DeclareMathOperator{\id}{id}          
\DeclareMathOperator{\HH}{H}

\DeclareMathOperator{\res}{res}

\newcommand{\Hb}{{\rm H}_{\rm b}}

\newcommand{\Linfty}{L^\infty}

\newcommand{\Binfty}{\scL^\infty}

\newcommand{\lmod}{\!\setminus\!}

\DeclareMathOperator{\GL}{GL}         
\DeclareMathOperator{\SL}{SL}         
\DeclareMathOperator{\PSL}{PSL}       
\DeclareMathOperator{\OO}{O}               
\DeclareMathOperator{\SO}{SO}             

\DeclareMathOperator{\Sp}{Sp}              

\newcommand{\qand}{\quad \mathrm{and} \quad}

\AtBeginDocument{\addtocontents{toc}{\protect\setlength{\parskip}{0pt}}}

\begin{document}

\title[Bounded cohomology of complex Lie groups of classical type]{The degree-three bounded cohomology of complex \vspace{3pt}\\ Lie groups of classical type \vspace{-2pt}}
\author{Carlos De la Cruz Mengual \vspace{-6pt}}
\address{Faculty of Electrical and Computer Engineering \\ Technion, Haifa, Israel}
\email{c.delacruz@technion.ac.il}

\begin{abstract}
We establish Monod's isomorphism conjecture in degree-three bounded cohomology for every complex simple Lie group of classical type. Our main ingredient is a bounded-cohomological stability theorem with an optimal range in degree three that we bootstrap from previous stability results by the author and Hartnick. The bootstrapping procedure relies on the occurrence in our setting of a variant of the recently observed phenomenon of secondary stability in the sense of Galatius--Kupers--Randal-Williams.  \vspace{-30pt}
\end{abstract}

\maketitle

\section{Introduction}

\subsection{Statement of main result}
This article is the third one in a series concerned with the explicit computation of the continuous bounded cohomology of Lie groups via stabilization techniques (see prequels \cite{DMH0,DM+Hartnick}). The theory of continuous bounded cohomology was developed in the early 2000s by Burger and Monod \cite{Burger-Monod2, Monod-Book} as a topological refinement of the celebrated notion for discrete groups (conceived by Johnson, Trauber, and Gromov \cite{Gromov} independently), with powerful applications in geometry,
dynamics and rigidity theory
\cite{Campagnolo-etal, Frigerio, Monod-Book}. 

Despite its power, the computation of continuous bounded cohomology remains a notorious challenge. In the setting of Lie groups, the problem can be reduced to the connected semisimple case, whose behavior is predicted by the prominent 

\begin{conj*} \label{conjecture}
For any connected semisimple Lie group $G$ with finite center, the comparison map $c^\ast:\Hb^\ast(G) \to \HH^\ast(G)$ is an isomorphism in every degree. 
\end{conj*}

Here we denote by $\HH^\ast(G)$ the continuous group cohomology of $G$ with coefficients on the trivial $G$-module $\bbR$, and by $\Hb^\ast(G)$ the corresponding continuous bounded cohomology. The comparison map $c^\ast$ is defined by forgetting the boundedness of a class. The isomorphism conjecture was formulated in this form by Monod in his ICM address \cite{Monod-Survey} in 2006. The surjectivity question, posed earlier by Dupont \cite{Dupont}, has been answered affirmatively in a variety of contexts, e.g. \cite{HartOtt2,LS}. We emphasize that the continuous cohomology of Lie groups is well understood in virtue of the van Est isomorphism (see e.g. \cite[Cor. XI.5.6]{BW} or \cite{Stas}) and classical results on the cohomology of compact symmetric spaces (see e.g. \cite{GHV,Toda-Mimura}). 

Beyond the trivial degrees zero and one, the isomorphism conjecture has been established fully in degree two by Burger and Monod \cite{Burger-Monod1}. However, in degrees $\geq 3$, the conjecture has seen scarce progress in the last two decades. In degree three, it has been verified for the groups 
$\SL_n(\bbR)$ \cite{Burger-Monod3,Monod-Stab} and $\SL_n(\bbC)$ \cite{Bloch,Monod-Stab,Goncharov,BBI}. In higher degrees, our lack of understanding of the conjecture is remarkable, only known to hold in degree four for $\SL_2(\bbR)$ \cite{HartOtt}. 

The goal of this paper is to expose further evidence in favor of the isomorphism conjecture.

\begin{thmx} \label{main_thm}
The isomorphism conjecture holds in degree three for every connected simple complex Lie group $G$ of classical type. In particular, we have $\Hb^3(G) \cong \HH^3(G) \cong \bbR$. 
\end{thmx}

This statement encompasses the four classical complex families $\mathbf{A}_r, \, \mathbf{B}_r,\,\mathbf{C}_r,\,\mathbf{D}_r$, which consist of the connected Lie groups locally isomorphic to 
\begin{equation} \label{eq:families}
	\SL_{r+1}(\bbC), \quad \SO_{2r+1}(\bbC), \quad \Sp_{2r}(\bbC), \quad \SO_{2r}(\bbC).
\end{equation}
respectively. In each case, $r$ is the rank of the group, and all are simple Lie groups whenever $r \geq 1$, except for $\SO_2(\bbC)$ and $\SO_4(\bbC)$ (abelian and with two simple factors, respectively). 

Non-vanishing results for the continuous bounded cohomology of Lie groups find important applications in the study of representations of fundamental groups of manifolds, and in particular, of their moduli spaces and deformations. The case of surface group representations with Hermitian targets, studied extensively in existing literature, serves as the most prominent example of this claim (see, e.g., \cite{BIW}). The relevant cohomological input in this setting is the non-vanishing of $\Hb^2$ for non-compact Lie groups of Hermitian type. 

A bounded-cohomological approach has been also applied successfully in the study of 3-manifold group representations into $\PSL_n(\bbC)$---a group with non-vanishing $\Hb^3$---and has been instrumental in proofs of interesting rigidity phenomena (see, e.g., \cite{BBI,Farre}). It is our hope that \autoref{main_thm} be the foundation of future investigations of representation varieties of 3-manifold groups into more general complex Lie groups.

\subsection{The role of bounded-cohomological stability}
Each of the four sequences of groups indexed by $r$ and listed in \eqref{eq:families} can be organized as increasing chains of inclusions
\begin{equation} \label{inclusions_intro}
	S_0 \overset{\iota_0}{\hookrightarrow} S_1 \overset{\iota_1}{\hookrightarrow} S_2 \overset{\iota_2}{\hookrightarrow} \cdots \hookrightarrow S_r \overset{\iota_r}{\hookrightarrow} S_{r+1} \hookrightarrow \cdots
\end{equation}
where $S_r$ is the group of rank $r$ and the inclusions $\iota_r$ are block embeddings in appropriate bases. The following theorem concerning the degree-three continuous bounded cohomology of the groups in these families plays a key role in the proof of \autoref{main_thm}. 

\begin{thmx} \label{new_stability}
	The embeddings $\iota_r$ in each of the classical families induce the isomorphisms
\begin{equation*} \begin{array}{ccccccccccc}
	(\mathbf{A}_r)& \qquad & \Hb^3(\SL_2(\bbC)) &\xleftarrow{\ \cong} & \Hb^3(\SL_3(\bbC)) &\xleftarrow{\ \cong} & \cdots & \xleftarrow{\ \cong} & \Hb^3(\SL_{r+1}(\bbC)) & \xleftarrow{\ \cong} & \cdots \\
	(\mathbf{B}_r)& \qquad & \Hb^3(\SO_3(\bbC)) &\xleftarrow{\ \cong} & \Hb^3(\SO_5(\bbC)) &\xleftarrow{\ \cong} & \cdots & \xleftarrow{\ \cong} & \Hb^3(\SO_{2r+1}(\bbC)) & \xleftarrow{\ \cong} & \cdots \\
	(\mathbf{C}_r)& \qquad & \Hb^3(\Sp_2(\bbC)) &\xleftarrow{\ \cong} & \Hb^3(\Sp_4(\bbC)) & \xleftarrow{\ \cong} & \cdots & \xleftarrow{\ \cong} & \Hb^3(\Sp_{2r}(\bbC)) & \xleftarrow{\ \cong} & \cdots \\
	(\mathbf{D}_r)& \qquad & \Hb^3(\SO_6(\bbC)) & \xleftarrow{\ \cong} & \Hb^3(\SO_8(\bbC)) &\xleftarrow{\ \cong} & \cdots & \xleftarrow{\ \cong} & \Hb^3(\SO_{2r}(\bbC)) & \xleftarrow{\ \cong} & \cdots  \\[5pt]
\end{array} \end{equation*}
\end{thmx}

The property that the embeddings $\iota_r$ in a chain of inclusions as \eqref{inclusions_intro} induce isomorphisms 
\begin{equation} \label{induced_maps}
	\Hb^q(S_{r_0}) \xleftarrow{\ \cong} \ \Hb^q(S_{r_0+1}) \xleftarrow{\ \cong} \ \Hb^q(S_{r_0+2}) \xleftarrow{\ \cong} \ \cdots
\end{equation}
 for every degree $q \in \bbN$, starting from an index $r_0 = r_0(q)$, is known as \emph{bounded-cohomological stability} (abbrv. \emph{bc-stability}). An assignment $q \mapsto r_0(q)$ as above is called \emph{bc-stability range}. 

\autoref{new_stability} settles the \emph{optimal} (i.e. lowest possible) bc-stability range $r_0(3)$ in degree $q=3$ for the four complex classical families. In the case of the $\mathbf{A}_r$ family, \autoref{new_stability} was proved by Bucher, Burger, and Iozzi \cite[Thm. 2]{BBI}, relying on work by Monod \cite{Monod-Stab}. Establishing the statement for the remaining three families is the crux of this article. 

The isomorphism conjecture, if true, would imply bc-stability for all classical families of Lie groups---for which continuous cohomology is known to stabilize. Our derivation of \autoref{main_thm} from \autoref{new_stability} is a sort of converse of this implication, based on the following principle. If a classical family $(S_r)_{r \in \bbN}$ is bc-stable with range $r_0(q)$, then proving the degree-$q$ isomorphism conjecture for all the groups $(S_r\mid r \geq r_0(q))$ within the range of bc-stability is reduced to verifying it only for $S_{r_0(q)}$. The specific feature that favors the application of this principle in our setting is that the range $r_0(3)$ established by \autoref{new_stability} is low enough to reach groups whose $\Hb^3$ is well understood. Indeed, all the groups at the base of the respective isomorphism chains for the families $\mathbf{B}_r, \mathbf{C}_r, \mathbf{D}_r$ are of type $\mathbf{A}_r$ in disguise.

\subsection{A secondary stability phenomenon} \label{intro_proof}
To explain the structure of our proof of \autoref{new_stability} for the families $\mathbf{B}_r, \mathbf{C}_r, \mathbf{D}_r$, we revisit the argument in Bucher--Burger--Iozzi \cite{BBI} for the $\mathbf{A}_r$ case. 

In \cite{Monod-Stab}, Monod proved the bc-stability of the family $\mathbf{A}_r$, with an addition that ensures a few extra injections---instead of isomorphisms---beyond the regime of bc-stability. In degree three:
\begin{equation*} 
	\Hb^3(\SL_2(\bbC)) \hookleftarrow \Hb^3(\SL_3(\bbC)) \hookleftarrow \Hb^3(\SL_4(\bbC)) \cong \Hb^3(\SL_5(\bbC)) \cong \cdots \cong \Hb^3(\SL_{r+1}(\bbC)) \cong \cdots  \vspace{-1pt}
\end{equation*}
Since $\Hb^3(\SL_2(\bbC))\cong \bbR$ (see \cite{Bloch,Burger-Monod3}), one deduces that $\Hb^3(\SL_n(\bbC))$ must have \emph{at most} dimension one for $n \geq 3$. That the dimension is \emph{at least} one is the content of  \cite{BBI}, which exhibits a non-zero class in $\Hb^3(\SL_n(\bbC))$ based on a construction by Goncharov \cite{Goncharov}. 

In analogy to \eqref{induced_maps}, we call \emph{weak} bc-stability the property that the embeddings $\iota_r$ in \eqref{inclusions_intro} induce injections starting from an index $r_1 = r_1(q)$ for every degree $q \in \bbN$. The (weak) bc-stability of the complex $\mathbf{B}_r, \mathbf{C}_r, \mathbf{D}_r$---among other classical families---was first proved by Hartnick and the author in \cite[Thm. A, Cor. B]{DM+Hartnick}. In recent work, Kastenholz and Sroka \cite{KasSr} improved the range of weak bc-stability given in \cite{DM+Hartnick} for those families to $r_1(q)=2q+2$. 

As opposed to Monod's weak bc-stability result \cite{Monod-Stab} for $\mathbf{A}_r$, the ranges established in the references \cite{DM+Hartnick,KasSr} for the complex families $\mathbf{B}_r, \mathbf{C}_r, \mathbf{D}_r$ do not yield a reduction to a group whose $\Hb^3$ is understood. Thus, our first step towards \autoref{new_stability} is the following range improvement in degree three.

\begin{klem*} \hypertarget{bootstrap}
The block embeddings $\iota_r$ induce injections 
\begin{equation*} \begin{array}{ccccccccccc}
	&&\Hb^3(\OO_3(\bbC)) &\hookleftarrow & \Hb^3(\OO_5(\bbC)) &\hookleftarrow & \cdots & \hookleftarrow & \Hb^3(\OO_{2r+1}(\bbC)) & \hookleftarrow & \cdots \\[2pt]
	&&\Hb^3(\Sp_2(\bbC)) &\hookleftarrow & \Hb^3(\Sp_4(\bbC)) & \hookleftarrow & \cdots & \hookleftarrow & \Hb^3(\Sp_{2r}(\bbC)) & \hookleftarrow & \cdots \\[2pt]
	\Hb^3(\OO_4(\bbC)) & \hookleftarrow & \Hb^3(\OO_6(\bbC)) &\hookleftarrow & \Hb^3(\OO_8(\bbC)) &\hookleftarrow & \cdots & \hookleftarrow & \Hb^3(\OO_{2r}(\bbC)) & \hookleftarrow & \cdots 
\end{array} \end{equation*}
\end{klem*}

\begin{rem}
The omission of the determinant-one condition in the orthogonal groups in the \hyperlink{bootstrap}{Key Lemma} is on purpose. Upgrading it to determinant-one groups is a technical point that we will treat later. We ask from the reader to ignore that at this stage, and assume for the sake of the introduction that the \hyperlink{bootstrap}{Key Lemma} has been proven for the special orthogonal groups.
\end{rem}

A remarkable aspect of the \hyperlink{bootstrap}{Key Lemma} is the next fact that lies at the heart of its proof, given in Section \ref{sec:bootstrap}. We will prove that if \emph{any} of the block embeddings $\iota_r$ for the families $\mathbf{B}_r, \mathbf{C}_r, \mathbf{D}_r$ induces an injection at the level of $\Hb^3$, then necessarily \emph{all} the $\iota_r$'s up to the ranges produced by the \hyperlink{bootstrap}{Key Lemma} must induce injections. In this sense, any arbitrary range (e.g., \cite{DM+Hartnick,KasSr}) is as good as any other, for either of them will allow us to bootstrap the optimal range \emph{a posteriori}. 

The proof of this intriguing fact relies on a principle reminiscent of \emph{secondary stability} in the sense of Galatius--Kupers--Randal-Williams \cite{Galatius}. A subject of ongoing research in the realm of ordinary homology, secondary stability is the event in which ``the failure of stability is itself stable''. In our case, failure of weak bc-stability is measured by the kernels of the maps induced by the embeddings $\iota_r$. We prove the stability of those kernels based on a parametrization via cross-ratios of configuration spaces of isotropic projective points, exposed in Section \ref{homog}.

\subsection{Non-triviality of $\Hb^3$ and Gromov norms}
In analogy with the argument for $\mathbf{A}_r$ given above, it remains to show the non-triviality of $\Hb^3$ for the relevant groups. We will prove:  

\begin{thmx} \label{restriction_inj}
For any $r \in \bbN$, let $S_r$ be either $\SO_{2r+1}(\bbC),\, \Sp_{2r}(\bbC)$, or $\SO_{2r}(\bbC)$, and let $n=n(r)$ be the dimension of the standard $S_r$-representation, so that $S_r < \SL_{n}(\bbC)$. Then, the restriction
\begin{equation} \label{eq:restriction_intro}
	\res_r: \Hb^3(\SL_{n}(\bbC)) \to \Hb^3(S_r)
\end{equation}
is an isomorphism, except if $S_r = \SO_{2r}(\bbC)$ and $r \in \{1,2\}$. 
\end{thmx}

Since $\Hb^3(\SL_n(\bbC))$ is 1-dimensional for every $n \geq 2$, the injectivity of the map \eqref{eq:restriction_intro} implies the desired non-triviality statement. For groups of types $\mathbf{B}_r$ and $\mathbf{C}_r$, this follows from the next proposition, which is of independent interest, for it also elucidates the norm structure of $\Hb^3$.  


\begin{propx} \label{thm:norms}
For $r \in \bbN$, if $S_r$ in \autoref{restriction_inj} is $\SO_{2r+1}(\bbC)$ or $\Sp_{2r}(\bbC)$, then $\res_r$ is isometric. 
\end{propx}

The proofs of \autoref{thm:norms} and Theorems \ref{new_stability} and \ref{restriction_inj} for the families $\mathbf{B}_r$ and $\mathbf{C}_r$, are given in Section \ref{sec:BC}. On the other hand, \autoref{thm:norms} does not hold in the $\mathbf{D}_r$ case (see \autoref{norm_low}). Thus, we give a separate argument for Theorems \ref{new_stability} and \ref{restriction_inj} for that family in Section \ref{sec:D}.

\subsection{Outlook}
In Section \ref{subsec:norm_D}, we discuss the norm structure of $\Hb^3$ for the classical complex Lie groups, and provide a conjecture in the $\mathbf{D}_r$ case, which remains open. 

With little extra bookkeeping of the various types of Levi factors of parabolic subgroups, it should be possible to apply our ideas to prove the degree-three isomorphism conjecture for the real split Lie groups $\SO_+(r+1,r)$, $\Sp_{2r}(\bbR)$, and $\SO_+(r,r)$. 
In general, we expect the situation in degree three for classical groups to be fully determined by what occurs in rank one. 

Finally, an exciting and potentially fruitful research avenue is the further usage of secondary stability techniques in bounded cohomology. This could serve for optimizing existing stability ranges, and, eventually, lead to new computations of bounded cohomology in higher degrees. 

\subsection*{Acknowledgments}
I extend my gratitude to Marc Burger for suggesting the problem that led to this work. I also thank Uri Bader for many helpful discussions. I am indebted to Tobias Hartnick for the collaboration that led to the two prequels \cite{DMH0,DM+Hartnick}, essential to this article, and grateful for his valuable feedback throughout the development of this project. I thank Manuel Krannich for pointing me to the notion of secondary stability and for his comments on a previous version of this article. Finally, I thank Nicolas Monod for suggesting the transit over discrete bounded cohomology in the proof of \autoref{lem_ind_base}.   

This work contains results of my doctoral thesis \cite{DM-Thesis}, and was completed during a postdoctoral fellowship at the Weizmann Institute of Science, Israel. I acknowledge support of the Swiss National Science Foundation, through the Early Postdoc.Mobility Grant No. 188010.

\section{Notation and standing assumptions} \label{sec:notation}
\subsection{Notation on formed spaces} \label{subsec:notation}
In the rest of this article, we will adhere to the notation from \cite[\S 2]{DM+Hartnick}, which we summarize below. The notation will be slightly simplified, as we only consider symmetric or antisymmetric formed spaces over $\bbC$. 

Let $\varepsilon \in \{\pm 1\}$. An $\varepsilon$\emph{-symmetric formed space} is a pair $(V,\omega)$ of a finite-dimensional $\bbC$-vector space $V$ and an $\varepsilon$-symmetric bilinear form $\omega: V \times V \to \bbC$. 
We write $q(v)\coloneqq \omega(v,v)$ for the quadratic form. The rank of $(V,\omega)$ is the maximal dimension of a totally isotropic subspace of $V$. Any non-degenerate $\varepsilon$-symmetric space $(V,\omega)$ of dimension $n$ splits as the direct sum of formed spaces
\vspace{-2pt}
\begin{equation} \label{decomposition}
	V \cong \clH_\varepsilon^{\oplus r} \oplus \clE^{\oplus d},\vspace{-2pt}
\end{equation}
where $r$ is the rank of $(V,\omega)$, $\clH_\varepsilon$ is the $\varepsilon$-hyperbolic plane, $\clE$ is the Euclidean line, and $d = n - 2r \in \{0,1\}$. Furthermore, if $\varepsilon = -1$, then $d=0$. We denote by $(V_{r}^{\varepsilon,d},\omega_{r}^{\varepsilon,d})$ the formed space on the right-hand side of \eqref{decomposition}. An ordered basis $\scB_{r}^{\varepsilon,d}=\{e_r,\ldots,e_1,dh,f_1,\ldots,f_r\}$ of $V_{r}^{\varepsilon,d}$ is an adapted basis of $(V_{r}^{\varepsilon,d},\omega_{r}^{\varepsilon,d})$ if $\langle h\rangle \cong \clE$ and $\langle e_i,f_i \rangle \cong \clH$ for every $i$. The matrix
\begin{equation*} \label{matrix_J}
J_{r}^{\varepsilon,d}:=\begin{pmatrix}
			 0 & 0 & \varepsilon Q_r\\[-2pt] 
			 0 & 1_d & 0 \\[-2pt]
			 Q_r & 0 & 0
			\end{pmatrix} \in {\rm M}_{n}(\bbC),
\end{equation*}
represents the form $\omega_{r}^{\varepsilon,d}$ in the basis $\scB_{r}^{\varepsilon,d}$, where $Q_r \in {\rm M}_{r}(\bbC)$ is the matrix with 1's on its antidiagonal and zero elsewhere (the middle row and column are omitted if $d=0$). 

Let $G_{r}^{\varepsilon,d}$ be the automorphism group of $(V_{r}^{\varepsilon,d},\omega_{r}^{\varepsilon,d})$ and $S_{r}^{\varepsilon,d}  < G_{r}^{\varepsilon,d}$ its determinant-one subgroup. Both $G_{r}^{\varepsilon,d}$ and $S_{r}^{\varepsilon,d}$ are complex Lie subgroups of $\GL(V_{r}^{\varepsilon,d})$. Concretely, \vspace{-3pt}
\begin{equation*} 
	G_{r}^{\varepsilon,d} = \begin{cases}
		\OO_{2r+1}(\bbC), & (\varepsilon,d) = (+1,1) \\[-2pt]
		\Sp_{2r}(\bbC), & (\varepsilon,d) = (-1,0) \\[-2pt]
	\OO_{2r}(\bbC), & (\varepsilon,d) = (+1,0)
\end{cases}
\end{equation*}
Observe that \vspace{-3pt}
\begin{equation} \label{eq:splitting}
	G_r^{(+1,1)} \cong \{\pm I\} \times S_r^{(+1,1)} \qand G_r^{(-1,0)} = S_r^{(-1,0)}.\vspace{-2pt}
\end{equation} 

For fixed $d$, the inclusions $\scB_{r}^{\varepsilon,d} \hookrightarrow \scB_{r+1}^{\varepsilon,d}$ of adapted bases gives rise to the embeddings \vspace{-3pt}
\begin{equation} \label{eq:def_iotar}
	G_{0}^{\varepsilon,d} \overset{\iota_0}{\hookrightarrow} G_{1}^{\varepsilon,d} \overset{\iota_1}{\hookrightarrow} G_{2}^{\varepsilon,d} \hookrightarrow \cdots \hookrightarrow G_{r}^{\varepsilon,d} \overset{\iota_r}{\hookrightarrow} G_{r+1}^{\varepsilon,d} \hookrightarrow \cdots \vspace{-2pt}
\end{equation}
of automorphism groups and of their corresponding determinant-one subgroups $S_r^{\varepsilon,d}$. We abuse notation and denote by $\iota_r$ also the embedding $S_{r}^{\varepsilon,d} \hookrightarrow S_{r+1}^{\varepsilon,d}$ (see also \eqref{inclusions_intro}). 

\subsection{Standing assumptions} \label{subsec:assumptions}
Throughout the article, we will keep a fixed choice of a pair \vspace{-2pt}
\begin{equation} \label{eq:eps_d}
	(\varepsilon,d) \in \{(+1,1),(-1,0),(+1,0)\}, \vspace{-2pt}
\end{equation}
which is equivalent to the choice of the classical family $\mathbf{B}_r, \mathbf{C}_r, \mathbf{D}_r$, respectively. We omit the indices $\varepsilon,d$ whenever possible and write $V_r$, $G_r$, $S_r$, etc. instead of $V_{r}^{\varepsilon,d}$, $G_{r}^{\varepsilon,d}$, $S_{r}^{\varepsilon,d}$, etc. Finally, we set $n(r) \coloneqq \dim(V_r)$, and
\begin{equation*} 
	\quad r_0 \coloneqq \begin{cases}
		3, & \mbox{ if } (\varepsilon,d) = (+1,0), \\[-4pt] 
		1, & \mbox{ otherwise,}
	\end{cases} \quad
	r_1 \coloneqq \begin{cases}
		2, & \mbox{ if } (\varepsilon,d) = (+1,0), \\[-4pt]
		1, & \mbox{ otherwise.}
	\end{cases}
\end{equation*}

\subsection{Re-statement of main results}
In the notation fixed in this section, our main results admit the following compact re-statements (omitting the $\mathbf{A}_r$ case from Theorems \hyperlink{main_thm_bis}{\ref{main_thm}} and \hyperlink{new_stability_bis}{\ref{new_stability}}):

\begin{klemb*} \hypertarget{bootstrap_bis}
For $r \geq r_1$, the induced map $\iota_r^\ast: \Hb^3(G_{r+1}) \to \Hb^3(G_r)$ is injective. 
\end{klemb*}

\begin{rep_thm1} \hypertarget{main_thm_bis}
For $r \geq r_0$, the comparison map $c^3: \Hb^3(S_{r}) \to \HH^3(S_r)$ is an isomorphism. 
\end{rep_thm1}

\begin{rep_thm2} \hypertarget{new_stability_bis}
For $r \geq r_0$, the induced map $\iota_r^\ast: \Hb^3(S_{r+1}) \to \Hb^3(S_r)$ is an isomorphism. 
\end{rep_thm2}

\begin{rep_thm3} \hypertarget{restriction_inj_bis}
For $r \geq r_0$, the restriction $\res_r: \Hb^3(\SL_{n(r)}(\bbC)) \to \Hb^3(S_r)$ is an isomorphism. 
\end{rep_thm3}

\section{Configuration spaces of isotropic points} \label{homog}

\subsection{The variety of isotropic points} \label{subsec:points}
We will consider the following generalized flag variety. 

\begin{defn} \label{isotropic_cone}
For any $r \geq r_1$, we call the irreducible complex projective variety \vspace{-3pt}
\[
	\clP_r = \clP_{r}^{\varepsilon,d} \coloneqq \{[v] \in \bbP(V_r) \mid q(v) = 0\} \vspace{-3pt}
\]
the \emph{variety of isotropic points} of $V_r$. By Witt's lemma \cite[Thm. 7.4]{Taylor}, it is a compact homogeneous $G_r$-space  when endowed with the Hausdorff topology of its set of $\bbC$-points, and the Lebesgue class is its unique $G_r$-invariant measure class. The assumption $r \geq r_1$ excludes two pathological cases: the set $\clP_0$ is empty, and $\clP_{1}^{(+1,0)}$ is not irreducible as an algebraic set.
\end{defn}

Let $Q_r$ be the stabilizer of $[e_r] \in \clP_r$, a maximal parabolic subgroup of $G_r$ with Levi decomposition $Q_r = U_r \rtimes R_r$, where $U_r$ is unipotent and $R_r$ reductive. Represented in the basis $\scB_r$, the reductive factor $R_r \cong \bbC^\times \times G_{r-1}$ consists of the block-diagonal matrices 
\begin{equation*} 
	  \begin{pmatrix}
			\lambda & 0 & 0 \\[-2pt]
			 0 & g & 0 \\[-2pt] 
			 0 & 0 & \lambda^{-1}
		\end{pmatrix} 
\end{equation*}
with $\lambda \in \bbC^\times$ and $g \in G_{r-1}$. If $Q_r^-$ denotes the stabilizer of $[f_r] \in \clP_r$, then $Q_r \cap Q_r^- = R_r$.

We are interested in parametrizing generic orbits of the diagonal $G_r$-action on product spaces $\clP_r^k$, for $k \in \{3,4,5\}$. Genericity is understood in the measure-theoretical sense, and therefore, the products $\clP_r^k$ will be regarded as objects in the category of  \emph{Lebesgue $G_r$-spaces}. 

\begin{defn}
 We call \emph{Lebesgue space} a standard Borel space $X$ endowed with the measure class of a Borel probability measure. If $X$ admits a Borel action of a lcsc group $G$, and the measure class is $G$-invariant, we call $X$ a \emph{Lebesgue $G$-space}. Morphisms of Lebesgue ($G$-) spaces are equivalence classes (up to null sets) of measure-class preserving Borel ($G$-)maps. 
\end{defn}

\begin{exmpl}
An algebraic $\bbC$-variety $X$ with the regular action of an algebraic group $G$ is a Lebesgue $G$-space if we endow $X$ with its Lebesgue measure class. The Borel $\sigma$-algebra of the set of orbits $G \lmod X$ with respect to the Hausdorff topology on $X$ is standard (see \cite[Thm. 3.1.3]{Zimmer}), and hence $G \lmod X$ is a Lebesgue space. 
\end{exmpl}

If $v_0,\ldots,v_k$ are all non-zero vectors in $V_r$, we will write $[v_0,\ldots,v_k]$ as a shorthand for the tuple $([v_0],\ldots,[v_k])$ of points in $\bbP(V_r)$. We will say the tuple is in general position if so are the vectors $v_0,\ldots,v_k$ (or any choices of lifts). We define the subsets 
\begin{align} 
	\clP_r^{(k+1)}\, &\coloneqq \left\{[v_0,\ldots,v_k] \in \clP_r^{k+1} \mid \omega(v_i,v_j) \neq 0 \ \ \forall \ i\neq j\right\},\label{Pk_parentheses} \\
	\clP_r^{\{k+1\}}&\coloneqq \left\{\vec{v} \in \clP_r^{(k+1)} \mid \vec{v} \mbox{ is in general position} \right\} \label{Pk_brackets}
\end{align}
Both are $G_r$-invariant, $\frS_k$-invariant\footnote{As customary in homological algebra, we adopt the convention that $0 \in \bbN$. The symmetric group of the set $\{0,\ldots,k\} \subset \bbN$ will be denoted by $\frS_k$.}, Zariski-open subsets of $\clP_r^{k+1}$. In particular, they are also Hausdorff-open and co-null. Moreover, if 
$
	\partial_i\colon\clP_r^{k+1} \to \clP_r^k
$
are the ($G_r$-equivariant) face maps that forget the $i$-th entry for $i \in \{0,\ldots,k\}$, then $\partial_i(\clP_r^{(k+1)}) = \clP_r^{(k)}$ and $\partial_i(\clP_r^{\{k+1\}}) = \clP_r^{\{k\}}$.

Our first observation is that there is only one generic $G_r$-orbit of triples of points in $\clP_r$. \vspace{-5pt}

\begin{lem} \label{3transitive}
The group $G_r$ acts transitively on $\clP_r^{\{3\}}$ for any $r \geq r_1$. Moreover, the point stabilizer $M_{r} < G_r$ admits a continuous surjection $M_{r} \twoheadrightarrow M'_r$ with solvable kernel, where
\[
	M'_{r} = \begin{cases}
		G_{r-1}^{(+1,0)} 	& \mbox{ if } (\varepsilon,d) = (+1,1), \\[-3pt]
		G_{r-2} 	& \mbox{ if } (\varepsilon,d) = (-1,0),\\[-3pt]
		G_{r-2}^{(+1,1)}	& \mbox{ if } (\varepsilon,d) = (+1,0).
	\end{cases}
\]
\end{lem}

\begin{proof}
For $r \geq r_1$, we set 
\begin{align} 
	\phi_2 &\coloneqq e_r + f_r + x_0 \in \clP_r \quad \mbox{where} \label{phi2} \\[-3pt]
	x_0 &\coloneqq \begin{cases} 
		0 & \mbox{if } (\varepsilon,d) = (-1,0) \mbox{ and } r = 1, \\[-4pt]
		{\scriptstyle \sqrt{-2}} \, h & \mbox{if } (\varepsilon,d) = (+1,1) \mbox{ and } r = 1, \\[-4pt]
		e_{r-1} - f_{r-1} & \mbox{otherwise,} 
		\end{cases} \nonumber 
\end{align}
We show that for any $\mathbf{p} = [v_0,v_1,v_2] \in \clP_r^{\{3\}}$, there exists $g \in G_r$ such that $g \,\mathbf{p} = [e_r,\ f_r, \ \phi_2]$. Indeed, since $\langle v_0,v_1 \rangle < V_r$ is hyperbolic and $(\bbC^\times)^2 = \bbC^\times$, there exists $g_1 \in G_r$ such that
\[
	g_1 \, [v_0,v_1,v_2] = [e_r,f_r,e_r+f_r+x] 
\]
for some $x \in V_{r-1} < V_r$. Here $x = 0$ if and only if $n(r) = 2$, which holds only when $(\varepsilon,d) = (-1,0)$ and $r = 1$. If $n(r) > 2$, then $q(x) = -(1+\varepsilon) = q(x_0)$.  Now, by Witt's lemma, we find $g_2 \in G_{r-1} < G_r$ such that $g_2 x = x_0$, proving the claim.

Up to index two, the stabilizer $M_r$ is isomorphic to the point stabilizer of the $G_{r-1}$-action on 
\[
	\clQ = \{x \in V_{r-1} \mid q(x) = -(1+\varepsilon)\} 
\]
If $\varepsilon = -1$, then $e_{r-1} \in \clQ$ and, hence, $M_r$ projects onto $G_{r-2}$ with solvable kernel. If $\varepsilon = +1$ and $x \in \clQ$ is arbitrary, the orthogonal decomposition $V_{r-1} = \langle x \rangle \oplus \langle x \rangle^\perp$ implies that $M_r$ is isomorphic to $\Aut(\langle x\rangle^\perp,\omega|_{\langle x \rangle^\perp})$ up to index two. This yields the claim, since $\langle x \rangle^\top$ is non-degenerate and of codimension one in $V_{r-1}$. 
\end{proof}

\subsection{Cross-ratio coordinates on configurations of 4- and 5-tuples}
We will describe the quotients $G_r\lmod\clP_r^4$ and $G_r\lmod\clP_r^5$ as Lebesgue spaces in terms of the \emph{$\omega$-cross-ratios}, introduced by Kor\'anyi and Reimann \cite{KR} in the realm of complex hyperbolic geometry; see e.g. \cite{Goldman}.

\begin{defn} \label{def:cartan_cr}
For any $\mathbf{p} = (p_0,p_1,p_2,p_3) \in \clP_r^{(4)}$ with $p_i = [v_i]$, the ratios
\begin{align*}
	\CR_0(\mathbf{p}) &\coloneqq \frac{\omega(v_0,v_2)\cdot \omega(v_1,v_3)}{\omega(v_0,v_3)\cdot \omega(v_1,v_2)}, \\
	\CR_1(\mathbf{p}) &\coloneqq \CR_0(p_1,p_2,p_0,p_3)^{-1}, \\
	\CR_2(\mathbf{p}) &\coloneqq  \CR_0(p_2,p_0,p_1,p_3),
\end{align*}
are independent of the choices of line representatives $v_i$, thus giving rise to the  $G_r$-invariant, holomorphic functions $\CR_j:\clP_r^{(4)} \to \bbC^\times$, called $\omega$\emph{-cross-ratios}. We will regard them as $G_r$-invariant morphisms  
$\CR_j\colon\clP_r^4\to \hat\bbC$ of Lebesgue spaces. We also write $\CR_j^{-1} \coloneqq 1/\CR_j$. 
\end{defn}

The next lemma explains dependencies between $\omega$-cross-ratios, including the effect of the permutation of their arguments. It shows that the only $\omega$-cross-ratios associated to any generic choice of four points in $\clP_r$ are $\CR_0$, $\CR_1$, $\CR_2$, and that any two of them determine the third one. Its proof is a verification left as an exercise to the reader. 

\begin{lem} \label{perm_4}
 The following identities hold everywhere in $\clP_r^{(4)}$: 
\begin{enumerate}[label=\emph{(\roman*)}]
	\item $\CR_0 \,\circ\, (01) = \CR_0 \,\circ\, (23) = \CR_1^{-1} \,\circ\, (02)  = \CR_1^{-1} \,\circ\, (13) = \CR_2 \,\circ\, (12) = \CR_2 \,\circ\, (03)= \CR_0^{-1}$.
	\item $\CR_1 \,\circ\, (12) = \CR_1 \,\circ\, (03) = \CR_2^{-1} \,\circ\, (01) = \CR_2^{-1} \,\circ\, (23) = \CR_1^{-1}$.
	\item $\CR_2 \,\circ\, (02) = \CR_2 \,\circ\, (13) = \CR_2^{-1}$.
	\item $\CR_0 \cdot \CR_1^{-1} \cdot \CR_2 = \varepsilon$.
\end{enumerate}
\end{lem}

\begin{rem} \label{rem:classical}
Let us consider the rank-one examples (i.e. $r=1$) among the families treated in \autoref{new_stability}. If $(\varepsilon,d)=(-1,0)$, then the $\omega$-cross-ratios coincide with the classical cross-ratios on the complex projective line, since $G_1=\Sp_2(\bbC) = \SL_2(\bbC)$ and $\clP_1 = \bbP(\bbC^2)$. We will continue to refer to this as the \emph{classical setting}. \autoref{3transitive} recovers the well-known fact that any triple in $\clP_1^{\{3\}}$ is in the orbit of $(\infty,0,1) \in \hat\bbC^3$. As usual, we identify $\bbP(\bbC^2)$ with $\hat\bbC$ by 
\begin{equation} \label{eq:extended_complex}
	[f_1] \equiv \infty, \qand [e_1 + z \cdot f_1] \equiv z \ \mbox{for any } z \in \bbC. 
\end{equation}
Recall that for every $\mathbf{p} \in \clP_1^{\{4\}}$, the cross-ratio $\CR_0(\mathbf{p})$ is the only number $a \in \bbC\smallsetminus\{0,1\}$ such that $G_1 \cdot \mathbf{p} = G_1 \cdot (\infty,0,1,a)$. In the language of Lebesgue spaces, this means that the map $\CR_0$ descends to a isomorphism $G_1 \lmod \clP_1^4 \cong \hat\bbC$. In fact, due to the low dimension of $V_1$ in the classical setting, just one of the $\omega$-cross-ratios determines the value of the remaining two.

The only other rank-one example appears for $(\varepsilon,d) = (+1,1)$, which gives $G_1 = \OO_3(\bbC)$. As in the classical setting, we have an isomorphism $G_1\lmod \clP_1^4 \cong \hat\bbC$.
\end{rem}   

In contrast to the rank-one situation, one single $\omega$-cross-ratio does not suffice to parametrize generic 4-tuples in higher-rank. However, two of them are enough.

\begin{prop} \label{param4}
For any $r \geq 2$, the $G_r$-invariant morphism $\pi_3: \clP_r^4 \to \hat\bbC^2$ of Lebesgue spaces defined as
$
	 \pi_3 \coloneqq (\CR_1,\,\CR_2)
$
descends to an isomorphism $G_r\lmod \clP_r^4 \cong \hat\bbC^2$. 
\end{prop}

We explain now how the $\omega$-cross-ratios parametrize generic orbits of 5-tuples of points in $\clP_r$. 
\begin{defn} \label{alpha,beta,gamma}
For $j \in \{0,1,2\}$, we define morphisms $\alpha_j,\, \beta_j,\, \gamma_j\colon \clP_r^5 \to \hat\bbC$ 
\begin{align*} 
	\alpha_j(\mathbf{p}) &\coloneqq \CR_j \circ\,\partial_4(\mathbf{p}) = \CR_j(p_0,p_1,p_2,p_3) , \\[-3pt]
	\beta_j(\mathbf{p}) &\coloneqq \CR_j \circ\,\partial_3(\mathbf{p}) = \CR_j(p_0,p_1,p_2,p_4), \\[-3pt]
	\gamma_j(\mathbf{p}) &\coloneqq \CR_j \circ\,\partial_2(\mathbf{p}) = \CR_j(p_0,p_1,p_3,p_4),
\end{align*}
for $\mathbf{p}=(p_0,p_1,p_2,p_3,p_4) \in \clP_r^{(5)}$, where $\partial_i:\clP_r^5 \to \clP_r^4$ denote the usual face operators. As in \autoref{def:cartan_cr}, we set $\alpha_j^{-1} \coloneqq 1/\alpha_j$ and define similarly $\beta_j^{-1}$, $\gamma_j^{-1}$. 
\end{defn}

We can express in terms of the $\omega$-cross-ratios from \autoref{alpha,beta,gamma} all the remaining cross-ratios associated to a 5-tuple. This fact, recorded in the next lemma, is readily verified. 

\begin{lem} \label{lem:alpha,beta,gamma}
The following identities hold everywhere in $\clP_r^{(5)}$: \vspace{2pt}

\begin{minipage}[c]{0.36\linewidth}
\begin{enumerate}[label=\emph{(\roman*)}]
	\item $\CR_0\, \circ \, \partial_1 = \alpha_2 \beta_2^{-1}$
	\item $\CR_1\, \circ \, \partial_1 = \beta_1^{-1} \gamma_1$
	\item $\CR_2\, \circ \, \partial_1 = \varepsilon (\alpha_2^{-1} \beta_1^{-1} \beta_2\gamma_1)$
\end{enumerate}
\end{minipage}
\begin{minipage}[c]{0.36\linewidth}
\begin{enumerate}[label=\emph{(\roman*)}]
	\setcounter{enumi}{3}
	\item $\CR_0\, \circ \, \partial_0 = \alpha_1 \beta_1^{-1}$
	\item $\CR_1\, \circ \, \partial_0 = \varepsilon(\alpha_1 \alpha_2^{-1} \beta_1^{-1} \gamma_1)$
	\item $\CR_2\, \circ \, \partial_0 = \alpha_2^{-1} \gamma_1$
\end{enumerate}
\end{minipage}
\begin{minipage}[c]{0.28\linewidth}
\begin{enumerate}[label=\emph{(\roman*)}]
	\setcounter{enumi}{6}
	\item $\alpha_0 \gamma_0 = \beta_0$ \\ \\
\end{enumerate}
\end{minipage}
\end{lem}

We conclude from Lemmas \ref{perm_4} and \ref{lem:alpha,beta,gamma} that at most five cross-ratios suffice to describe generic orbits of 5-tuples. The next theorem states that five are also necessary as long as $\dim(V_r)$ is large. Recall that $r_1$ was defined as a function of $(\varepsilon,d)$ at the end of Subsection \ref{subsec:notation}.

\begin{prop} \label{param5}
For any $r \geq r_1 + 1$, the $G_r$-invariant morphism $\pi_4\colon\clP_r^5 \to \hat\bbC^5$ of Lebesgue spaces defined as 
$
\pi_4 \coloneqq (\alpha_1,\,\alpha_2,\,\beta_1,\,\beta_2,\,\gamma_1) 
$
descends to an isomorphism $G_r\lmod\clP_r^5 \cong \hat\bbC^5$. 
\end{prop}

\begin{rem} \label{low_rank_cases}
In rank one, the isomorphism $G_1\lmod \clP_1^5 \cong \hat\bbC^2$ holds for $(\varepsilon,d) \in \{(+1,1),(-1,0)\}$. For instance, in the classical setting (see \autoref{rem:classical}), any $\mathbf{p} \in \clP_1^{\{5\}}$ is in the $G_1$-orbit of the tuple $(\infty,0,1,a,b) \in \hat\bbC^5$ with $a=\alpha_0(\mathbf{p})$ and $b=\beta_0(\mathbf{p})$. In rank two for $(\varepsilon,d)=(+1,0)$, the low dimension of $V_2$ yields one extra relation, which forces the isomorphism $G_2\lmod \clP_2^5 \cong \hat\bbC^4$. 
\end{rem}

Both Propositions \ref{param4} and \ref{param5} will be essential in our proof of \hyperlink{bootstrap}{Key Lemma}. Their proof, in turn, consist of elementary, yet fairly lengthy and technical computations. For the sake of readibility, we have opted to present them in Appendix \ref{app:computation}.

\section{Proof of Key Lemma} \label{sec:bootstrap}

\subsection{Cohomological characterization of $\ker \iota_r^\ast$}
For any $r \geq 1$, we consider the complex 
\begin{equation}\label{LinftyComplex}
	0 \to \Linfty(\clP_r) \xrightarrow{\dd_r^{0}} \Linfty(\clP_r^2) \xrightarrow{\dd_r^{1}} \Linfty(\clP_r^3) \rightarrow \cdots
\end{equation}
of dual Banach $G_r$-modules, with $G_r$-equivariant coboundary maps defined by the alternating sum $\dd_r^k \coloneqq \sum_{i=0}^{k+1} (-1)^i \cdot \partial^i$. Here $\partial^i$ is the operator induced by the face maps $\partial_i: \clP_r^{k+1} \to \clP_r^k$. The complex \eqref{LinftyComplex} is \emph{acyclic} in the sense that the homology of its augmented complex 
\begin{equation}\label{LinftyComplexAug}
	0 \to \bbR \xrightarrow{\dd_r^{-1}} \Linfty(\clP_r) \xrightarrow{\dd_r^{0}} \Linfty(\clP_r^2) \xrightarrow{\dd_r^{1}} \Linfty(\clP_r^3) \rightarrow \cdots
\end{equation}
vanishes in every degree. This follows from the observation that integration over the first variable of $\Linfty(\clP_r^{k+1})$ with respect to a quasi-invariant probability measure on $\clP_r$ is a contracting homotopy of \eqref{LinftyComplexAug}. We denote by $\Hb^q(G_r \curvearrowright \clP_r)$ the $q$-th homology of the complex $\Linfty(\clP_r^{\bullet+1})^{G_r}$ of $G_r$-invariants of \eqref{LinftyComplex}, and call it the \emph{bounded cohomology of the action} $G_r \curvearrowright\clP_r$. 

In the rank-one cases, the action $G_1 \curvearrowright \clP_1$ is amenable, and hence, there exists a canonical isomorphism
$
	\Hb^\ast(G_1) \cong \Hb^\ast(G_1 \curvearrowright \clP_1);
$
see \cite[Thm. 7.5.3]{Monod-Book}. Beyond rank one, the isomorphism needs not hold since the action is no longer amenable. However, the bounded cohomology of that action---or any other $G_r$-action on an acyclic complex---may still contain interesting information about $\Hb^\ast(G_r)$. The precise relationship between $\Hb^\ast(G_r)$ and $\Hb^\ast(G_r \curvearrowright \clP_r)$ is given by a spectral sequence \cite[Prop. 2.15]{DMH0}, with which we prove the following

\begin{lem} \label{kernel_H3}
For any $r \in \bbN$, there exists a linear isomorphism
$
\ker \iota_{r}^\ast \ \cong \ \Hb^3(G_{r+1} \curvearrowright \clP_{r+1}).
$
\end{lem}

\begin{proof}
There exists a spectral sequence $\EE_\bullet^{\bullet,\bullet}$ that abuts to zero, with first-page terms 
\begin{equation*}
	\EE_1^{p,q} = \Hb^q(G_{r+1};\Linfty(\clP_{r+1}^{p})) 
\end{equation*}
 and differentials $\dd_1^{p,q}:\EE_1^{p,q}\to \EE_1^{p+1,q}$ induced by the operators $\dd_{r+1}^{p-1}$ \cite[Prop. 2.15]{DMH0}. Note first that $\EE_2^{p,0} \cong \Hb^{p-1}(G_{r+1} \curvearrowright \clP_{r+1})$ for all $p \geq 1$. Furthermore, as a consequence of the Eckmann--Shapiro lemma \cite[Prop. 10.1.3]{Monod-Book}, \autoref{3transitive}, and \cite[Corollary 8.5.2]{Monod-Book}, we have:
\begin{center}
\begin{tabular}{lcl} 
$\begin{array}{ll}
	\EE_1^{0,q} &= \Hb^q(G_{r+1}), \\[4pt]
	\EE_1^{1,q} &\cong \Hb^q(Q_{r+1}) \cong \Hb^q(G_{r}), \\[4pt]
	\EE_1^{2,q} &\cong \Hb^q(R_{r+1}) \cong \Hb^q(G_{r}), 
\end{array}$ && 
	$\EE_1^{3,q} \cong \Hb^q(M_{r+1})\cong \begin{cases}
			\Hb^q(G_{r}^{(+1,0)}) & \mbox{ if } (\varepsilon,d) = (+1,1), \\
			\Hb^q(G_{r-1}) & \mbox{ if } (\varepsilon,d) = (-1,0), \\
			\Hb^q(G_{r-1}^{(+1,1)}) & \mbox{ if } (\varepsilon,d) = (+1,0). 
		\end{cases}$
\end{tabular}
\end{center}
Here, $Q_{r+1}$, $R_{r+1}$, $M_{r+1}$ are as  defined in Subsection \ref{subsec:points}. For $p \leq 3$, we have
$
	\EE_2^{p,1}=\EE_1^{p,1} = 0$ 
and $\EE_2^{p,2}=\EE_1^{p,2} = 0.$
The latter equality holds from the isomorphism conjecture in degree two \cite[Lem. 6.1]{Burger-Monod1}. In fact, for any $(\varepsilon,d)$ as in \eqref{eq:eps_d} and any $r$, the group $G_r^{\varepsilon,d}$ is of non-Hermitian type, and therefore $\Hb^2(G_r^{\varepsilon,d}) \cong \HH^2(G_r^{\varepsilon,d})=0$. Finally, the map $\dd_1^{0,3}$ is conjugated by the isomorphisms above to the map $\iota_{r}^\ast$  (e.g. by \cite[Lem. 3.7]{DMH0}), and hence, $\EE_2^{0,3} \cong \ker \iota_{r}^\ast$. Summarizing, we have showed that the second page $\EE_2^{\bullet,\bullet}$ has terms as displayed below. \vspace{-5pt}
\begin{center}
\begin{tabular}{ p{10pt}|p{30pt} p{20pt} p{20pt} p{20pt} p{90pt}} 
  {\tiny 3}&$\ker \iota_{r}^\ast$ \\[-4pt] 
  {\tiny 2}&& 0 & 0 \\[-4pt]  
  {\tiny 1}&&& 0 & 0 \\[-4pt] 
  {\tiny 0}&&&&& $\Hb^3(G_{r+1}\curvearrowright \clP_{r+1})$ \vspace{3pt}\\[-2pt] 
 \hline
  &\quad {\tiny 0}&{\tiny 1}&{\tiny 2}&{\tiny 3}&\ \ \ \ \ \qquad{\tiny 4}
\end{tabular}
\end{center}
Blank spaces indicate that the terms are not relevant to our computation. 
Observe that all the displayed terms will remain unchanged until the fourth page $\EE_4^{\bullet,\bullet}$. Then, the arrow
\[
	\dd_4^{0,3}: \ker \iota_{r}^\ast \to \Hb^3(G_{r+1}\curvearrowright\clP_{r+1}) 
\]
must be an isomorphism, for otherwise the limits $\EE^3_\infty \neq 0$ or $\EE^4_\infty \neq 0$ would be non-zero. \end{proof}

\subsection{Secondary stability argument}
Propositions \ref{param4} and \ref{param5} enter the proof of the \hyperlink{bootstrap_bis}{Key Lemma \emph{bis}} in the form of the next statement, reminiscent of secondary stability in the sense of Galatius--Kupers--Randal-Williams \cite{Galatius}.

\begin{lem} \label{0row}
For every $r \geq r_1$, the linear isomorphism $\ker \iota_r^\ast \cong \ker \iota^\ast_{r_1}$ holds.
\end{lem}

\begin{proof}
In virtue of \autoref{kernel_H3}, we must prove the isomorphism
\[
	\Hb^3(G_{r+1} \curvearrowright \clP_{r+1}) \cong \Hb^3(G_{r_1+1} \curvearrowright \clP_{r_1+1}).
\]
By \autoref{3transitive}, the space of invariants $\Linfty(\clP_{r+1}^3)^{G_{r+1}}$ consists solely of constants, and hence, there are no coboundaries in $\Hb^3(G_{r+1} \curvearrowright \clP_{r+1})$. Let now $D_{r+1}^3$ be the bounded operator defined by the commutative diagram \vspace{-8pt}
\begin{equation*}
\xymatrixcolsep{3.5pc}
\xymatrixrowsep{2pc}
\xymatrix{\Linfty(\clP_{r+1}^4)^{G_{r+1}} \ar[r]^-{\dd_{r+1}^3} & \Linfty(\clP_{r+1}^5)^{G_{r+1}} \\
\Linfty(\hat\bbC^2) \ar@{-->}[r]^-{D_{r+1}^3} \ar[u]^{\pi_3^\ast}_\cong & \Linfty(\hat\bbC^5) \ar[u]_{\pi_4^\ast}^\cong}
\end{equation*}
Here, $\pi_3^\ast$ and $\pi_4^\ast$ are the isomorphisms induced by the morphisms of Lebesgue spaces from Propositions \ref{param4} and \ref{param5}, respectively. 
Observe that $\ker \dd_{r+1}^3 \cong \ker D_{r+1}^3$. 

The proof is completed upon showing that $D_{r+1}^3$ is independent of $r \geq r_1$. In fact, after applying \autoref{lem:alpha,beta,gamma}, the map $D_{r+1}^3$ can be expressed as \vspace{2pt}
\begin{equation} \label{higher_5term}
D^3_{r+1}f(\mathbf{x}) = f\!\left(\varepsilon\frac{a_1c_1}{a_2b_1}, \frac{c_1}{a_2}\right) - f\!\left(\frac{c_1}{b_1},\varepsilon \frac{b_2c_1}{a_2b_1}\right) + f\!\left(c_1,\varepsilon \frac{a_1b_2c_1}{a_2b_1}\right) - f(b_1,b_2) + f(a_1,a_2)\vspace{2pt}
\end{equation}
 for any $f \in \Linfty(\hat\bbC^2)$ and a.e. $\mathbf{x}=(a_1,a_2,b_1,b_2,c_1) \in \hat\bbC^5$.
\end{proof}

\begin{proof}[Proof of \hyperlink{bootstrap_bis}{Key Lemma \emph{bis}}]
The (weak) bc-stability of the complex classical families---proved in \cite[Thm. A]{DM+Hartnick}---implies that $\ker \iota_r^\ast = 0$ for large enough $r$. This fact combined with \autoref{0row} finishes the proof.
\end{proof}

In the classical setting (see \autoref{rem:classical}), the coboundary $\dd_1^3$ is conjugated to an operator 
$
D_1^3:\Linfty(\hat\bbC) \to \Linfty(\hat\bbC^2)
$
 in virtue of the isomorphisms discussed in Remarks \ref{rem:classical} and \ref{low_rank_cases}. The functional equation $D_1^3f = 0$ for $f \in \Linfty(\bbC)$ is known as the \emph{Spence--Abel functional equation}. Under a specific parametrization, it is given by \vspace{2pt}
\begin{equation} \label{spence_abel}
	f\!\left(\frac{a(1-b)}{b(1-a)}\right)
	-  f\!\left(\frac{1-b}{1-a}\right) +f\!\left(\frac{b}{a}\right) - f(b) +  f(a) = 0 \quad \mbox{ for a.e. } (a,b) \in \bbC^2,\vspace{2pt}
\end{equation}
and its only solution up to a constant factor is the \emph{Bloch--Wigner dilogarithm} $\clD$ (see e.g. \cite[\S 3]{Zagier}). The expression of $D_{\!\!\infty}^3 \coloneqq D_r^3$ given by \eqref{higher_5term} for any $r \geq r_1+1$ produces a higher-rank analogue of the Spence--Abel equation. A byproduct of the proof of the \hyperlink{bootstrap_bis}{Key Lemma} is the inexistence of non-trivial solutions for this equation.

\begin{cor} \label{thm:functionaleq}
Let $\varepsilon \in \{\pm 1\}$. If $f \in \Linfty(\hat\bbC^2)$ satisfies the functional equation $D_{\!\!\infty}^3f=0$ a.e., then $f$ is a.e. the identically zero function. \qed
\end{cor}

\section{Proof of Theorems 2 and 3 for the complex families $\mathbf{B}_r$ and $\mathbf{C}_r$} \label{sec:BC}

In Subsection \ref{subsec:assumptions}, we fixed a pair $(\varepsilon, d) \in \{(+1,1),(-1,0),(+1,0)\}$ as a standing assumption. In this section, we will only consider the cases $(+1, 1)$ and $(-1, 0)$, which correspond to the two complex classical families $\mathbf{B}_r$ and $\mathbf{C}_r$, respectively. In both cases, we have $r_0 = r_1 = 1$. 

We first prove that $\res_r: \Hb^3(\SL_{n(r)}(\bbC)) \to \Hb^3(S_r)$ is isometric for any $r \in \bbN$, which is the statement of \autoref{thm:norms}. We rely on the following result from \cite{BBI}. Below, we denote by $\Pi_n: \SL_2(\bbC) \to \SL_{n}(\bbC)$ the irreducible $n$-dimensional $\SL_2(\bbC)$-representation, and by $b_n$ the \emph{bounded Borel class}, defined in \cite{BBI}, which generates $\Hb^3(\SL_n(\bbC))$ as a vector space.  

\begin{thm}[{see \cite[Thm. 2]{BBI}}] \label{thm:irrep}
For any $n\geq 2$, the map $\Pi_n^\ast:\Hb^3(\SL_n(\bbC)) \to \Hb^3(\SL_2(\bbC))$ induced by $\Pi_n$ is an isometric isomorphism that maps the bounded Borel class $b_n$ to  
\[
	\Pi_n^\ast(b_n) = \frac{n(n^2-1)}6 \cdot b_2.
\]
In particular, the Gromov norm of $b_n$ is $\Vert b_{n} \Vert = n(n^2-1)/6 \cdot v_3,$ where $v_{\bbH^3} \approx 1.0149...$ is the maximal volume of an ideal tetrahedron in $\bbH^3$. 
\end{thm}

Consider now the irreducible representation $\Pi_{n(r)}:\SL_2(\bbC) \to \SL(V_r) \cong \SL_{n(r)}(\bbC)$ for $r \geq 1$. The parity of $n(r)$ depends on the choice of $(\varepsilon,d) \in \{(+1,1),(-1,0)\}$. It is important to note that (up to conjugation) $\Pi_{n(r)}$ preserves the form $\omega_r$ on $V_r$, which implies that its image is contained in $S_r$. We will use the notation $\bar\Pi_r: \SL_2(\bbC) \to S_r$ for the co-restriction of $\Pi_{n(r)}$. 

\begin{proof}[Proof of \autoref{thm:norms}]
Note that $\Pi_{n(r)}^\ast = \bar\Pi_r^\ast \,\circ\, \res_r$. Since $\Hb^3(\SL_{n(r)}(\bbC))$ is generated by the bounded Borel class $b_{n(r)}$, it suffices to show the equality $\Vert \res_r(b_{n(r)}) \Vert = \Vert b_{n(r)} \Vert$. Indeed, in virtue of \autoref{thm:irrep} and the fact that both $\res_r$ and $\bar\Pi_r^\ast$ are norm non-increasing, we have
\[
	\| b_{n(r)} \| = \| \Pi_{n(r)}^\ast (b_{n(r)}) \| = \| \bar\Pi_{r}^\ast (\res_r(b_{n(r)})) \| \leq \| \res_r(b_{n(r)}) \| \leq \| b_{n(r)} \|,
\]
completing the proof.
\end{proof}

For $(\varepsilon,d)$ as in this section, the restriction map $\Hb^\ast(G_r) \to \Hb^\ast(S_r)$ is an isomorphism as a result of \eqref{eq:splitting} and \cite[Cor. 8.5.2]{Monod-Book}. This implies immediately that \hyperlink{bootstrap_bis}{Key Lemma \emph{bis}} holds \emph{verbatim} for $S_r$ instead of $G_r$ in these two cases: 
\begin{lem} \label{klem_BC}
For $r \geq r_1 = 1$, the induced map $\iota_r^\ast: \Hb^3(S_{r+1}) \to \Hb^3(S_r)$ is injective. \qed
\end{lem}

Recall also that, depending on $(\varepsilon,d)$, the group $S_1$ equals $\Sp_2(\bbC)$ or $\SO_3(\bbC)$. The former group equals $\SL_2(\bbC)$ and the latter admits an isogeny $\SL_2(\bbC) \twoheadrightarrow \SO_3(\bbC)$, as seen in, e.g., \cite{vdWaerden,Tao}. Thus, $S_1$ is of type $\mathbf{A}_1$, and, in particular, $\dim \Hb^3(S_1) = 1$. 

\begin{proof}[Proof of \hyperlink{restriction_inj_bis}{Theorem \ref{restriction_inj} \emph{bis}} for $\mathbf{B}_r,\mathbf{C}_r$]
If $r \geq r_0 = 1$, then $n(r) \geq 2$, so that $\Vert b_{n(r)} \Vert > 0$. Together with \autoref{thm:norms}, this establishes the injectivity of $\res_r$. 
Surjectivity follows then from the fact that $\dim \Hb^3(S_r) \leq 1$ for all $r \geq 1$. This, in turn, is a consequence of \autoref{klem_BC} and the equality $\dim \Hb^3(S_1) = 1$. 
\end{proof}

\begin{proof}[Proof of \hyperlink{new_stability_bis}{Theorem \ref{new_stability} \emph{bis}} for $\mathbf{B}_r,\mathbf{C}_r$]
\hyperlink{restriction_inj_bis}{Theorem \ref{restriction_inj} \emph{bis}} implies that $\dim\Hb^3(S_r) = 1$ for all $r\geq 1$. Thus, the induced maps $\iota_r^\ast: \Hb^3(S_{r+1}) \to \Hb^3(S_r)$---injective by \autoref{klem_BC}---are isomorphisms. 
\end{proof} 

\section{Proof of Theorems 2 and 3 for the complex family $\mathbf{D}_r$} \label{sec:D}

In this section, we consider the case $(\varepsilon,d) = (+1,0)$, which corresponds to the $\mathbf{D}_r$ family, and recall that $n(r)=2r$, $r_0=3$, and $r_1=2$. Note that the images of the irreducible representations $\Pi_{2r}:\SL_2(\bbC) \to \SL_{2r}(\bbC)$ are not contained in $S_r=\SO_{2r}(\bbC)$ for any $r$. This limitation prevents the adaptation of the argument in Section \ref{sec:BC} to the $\mathbf{D}_r$ family.

We will establish \hyperlink{restriction_inj_bis}{Theorem \ref{restriction_inj} \emph{bis}} by induction on the rank $r$ without relying on norm considerations. The induction will be based on the following statement. 

\begin{lem} \label{lem_ind_base}
The restriction map $\res_2:\,\Hb^3(\SL_4(\bbC)) \to \Hb^3(\SO_4(\bbC))$ is injective.
\end{lem}

To prove \autoref{lem_ind_base}, we will realize $\res_2$ over the Furstenberg boundaries of the involved groups. Subsection \ref{subsec_cocycle} collects relevant background on the Goncharov--Bucher--Burger--Iozzi \emph{volume cocycle} of $\SL_n(\bbC)$, which represents the bounded Borel class $b_n \in \Hb^3(\SL_n(\bbC))$. Then, in Subsection \ref{furstenberg}, we will discuss the Furstenberg boundary of $\SO_4(\bbC)$.  

\subsection{The Goncharov--Bucher--Burger--Iozzi volume cocycle} \label{subsec_cocycle}
The content of this subsection is extracted entirely from the references \cite{Goncharov,BBI}. We abbreviate by $\bbP$ the complex projective line $\bbP(\bbC^2)$, isomorphic to the boundary of the hyperbolic 3-space $\bbH^3$. We let $\mathrm{Vol}_{\bbP}:\bbP^4 \to \bbR$ be the function that assigns to a 4-tuple $(p_0,p_1,p_2,p_3) \in \bbP^4$ the oriented volume of the ideal tetrahedron in $\bbH^3$ with vertices $p_i$. Up to rescaling, $\mathrm{Vol}_\bbP$ is the unique alternating, $\GL_2(\bbC)$-invariant, continuous bounded cocycle $\bbP^4 \to \bbR$
; see \cite{Bloch,Burger-Monod3}. 

We may write $\mathrm{Vol}_{\bbP}$ in terms of the Bloch--Wigner dilogarithm $\clD: \bbC \to \bbR$ (see \cite[\S 3]{Zagier}) as 
\begin{equation} \label{vol_dilog}
	\mathrm{Vol}_{\bbP} = \clD \, \circ \, \CR_0, 
\end{equation}
where $\CR_0$ is the classical cross-ratio on $\bbP$ (see \autoref{def:cartan_cr} and \autoref{rem:classical}). Under the identification $\bbP \cong \hat\bbC$ from \eqref{eq:extended_complex}, the equality \eqref{vol_dilog} corresponds to
$
	\mathrm{Vol}_{\bbP}(\infty,0,1,z) = \clD(z).
$ 
Thus, the cocycle equation for $\mathrm{Vol}_\bbP$ implies that $\clD$ satisfies Spence--Abel functional equation \eqref{spence_abel}, and the alternation of $\mathrm{Vol}_\bbP$ translates into the symmetries
\begin{equation} \label{symmetries_bw}
	\clD(z)=\clD\left(1-\frac1{z}\right) = \clD\left(\frac1{1-z}\right) = -\clD\left(\frac1{z}\right) = -\clD(1-z)=-\clD\left(\frac{-z}{1-z}\right). \vspace{3pt}
\end{equation}

A complete flag in $\bbC^n$ is a sequence $F^0 \!\subset\! F^1 \!\subset\! \cdots \!\subset\! F^{n-1} \!\subset\! F^{n}$ of subspaces $F^j \subset \bbC^n$ such that $\dim(F^j)=j$. The group $\SL_n(\bbC)$ acts transitively on the variety $\bbF_n$ of complete flags in $\bbC^n$, and the stabilizer of the standard flag $
	\{0\}\!\subset\! \langle e_1\rangle \!\subset\! \cdots \!\subset\! \langle e_1,\ldots,e_{n-1}\rangle \!\subset\! \bbC^n
$ is the upper-triangular subgroup $P_n < \SL_n(\bbC)$, a minimal parabolic subgroup. This identifies $\bbF_n$ with the Furstenberg boundary of $\SL_n(\bbC)$. Now, a complete \emph{affine} flag $(F,\vec{v})$ is a pair that consists of a complete flag $F$, and a tuple of spanning vectors $\vec v=(v^1,\ldots,v^{n})$ in the sense that \vspace{-3pt}
\[
	F^{j} = F^{j-1} \oplus \langle v^{j} \rangle \quad \mbox{for } j = 1,\ldots,n. \vspace{-3pt}
\] 
 We write $\bbF_{n,\mathrm{aff}}$ for the set of affine flags, which comes equipped with a projection $\bbF_{n,\mathrm{aff}} \twoheadrightarrow \bbF_n$.

\begin{defn} \label{def:sigma3}
Let $\sigma_3$ denote the collection of equivalence classes $[V;(v^0,\ldots,v^3)]$, where $V$ is a $\bbC$-vector space, $v^0,\ldots,v^3 \in V$ are vectors that span $V$. The equivalence $[V,(v^j)] \sim [W,(w^j)]$ holds if there exists a linear isomorphism
\begin{equation*} 
	\phi: V \xrightarrow{\cong \ } W \mbox{ with } \ \phi(v^j) = w^j \ \ \mbox{for all } j \in \{0,1,2,3\}.
\end{equation*}
Note that $\sigma_3$ is a set since every class $[V;(v^j)]$ admits a representative $[\bbC^m;(\bar{v}^j)]$ for $0 \leq m \leq 4$. 

Now, for any $J=(j_0,\ldots,j_3) \in \{0,\ldots,n-1\}^4$, we set $T_J:\bbF_{\mathrm{aff}}(\bbC^n)^4 \to \sigma_3$ as 
\begin{equation} \label{T_map}
	T_J\big((F_0,\vec v_0),\ldots,(F_3,\vec v_3)\big) \coloneqq \left[\frac{\langle F_0^{j_0+1},\ldots,F_3^{j_3+1}\rangle}{\langle F_0^{j_0},\ldots,F_3^{j_3}\rangle};\left(v_0^{j_0+1},\ldots,v_3^{j_3+1}\right)\right],
\end{equation}
where each $(F_i,\vec v_i) \in \bbF_{n,\mathrm{aff}}$ is such that $F_i = (F_i^0 \subset \cdots \subset F_i^n)$ and $\vec v_i = (v_i^1,\ldots,v_i^n)$, and the vectors in the right-hand side of \eqref{T_map} are regarded as classes modulo $\langle F_0^{j_0},\ldots,F_3^{j_3}\rangle$.
\end{defn}

\begin{defn} \label{borel cocycle}
Let $n \geq 2$, and let $\mathrm{Vol}: \sigma_3 \to \bbR$ be the map defined as 
\[
	\mathrm{Vol}\left[\bbC^m;(v^0,\ldots,v^3)\right] \coloneqq \begin{cases}
		\mathrm{Vol}_{\bbP}[v^0,\ldots,v^3] & \mbox{if } m = 2 \mbox{ and all } v^i \neq 0, \\[-3pt]
		0 & \mbox{otherwise}. 
	\end{cases}
\]
We let $\tilde B_n^J \colon \bbF_{n,\mathrm{aff}}^4 \to \bbR$ be the composition $\mathrm{Vol} \, \circ \, T_J$ for any $J \in \{0,\ldots,n-1\}^4$. This function is independent of the choices of spanning vectors, and thus, descends to a bounded Borel function $B_n^J: \bbF_n^4 \to \bbR$. Moreover, the sum \vspace{-4pt}
\[
	B_n \coloneqq \sum_{J} B_n^J \vspace{-4pt}
\]
descends to an alternating, $\GL_n(\bbC)$-invariant, bounded strict Borel cocycle $B_n:\bbF_n^4 \to \bbR$ (see {\cite[Cor. 13]{BBI}}). Observe that $B_2 = \mathrm{Vol}_{\bbP}$, so by analogy, we will call $B_n$ the (Goncharov--Bucher--Burger--Iozzi) \emph{volume cocycle}. The \emph{bounded Borel class} $b_n \in \Hb^3(\SL_n(\bbC))\cong \Hb^3(\GL_n(\bbC))$, mentioned in \autoref{thm:irrep} above, is defined as $b_n \coloneqq [B_n]$.
\end{defn}

Whilst $B_n$ is not everywhere continuous, it is so in the following generic locus. 

\begin{defn}
We say that a 4-tuple $(F_0,F_1,F_2,F_3) \in \bbF_n^4$ is in \emph{general position} if 
\[
	\dim\langle F_0^{j_0},\ldots,F_3^{j_3}\rangle = j_0 + \cdots + j_3
\]
as long as the right-hand side is at most $n$. We will denote by 
$\bbF_n^{\{4\}}$ the Zariski-open (hence Hausdorff-open and co-null) subset of $\bbF_n^4$ of 4-tuples in general position. 
\end{defn} 

\begin{lem} \label{gen_continuity}
For any $n \geq 2$ and $J =(j_0,\ldots,j_3) \in \{0,\ldots,n-1\}^4$, the restriction of $B_n^J$ to the open subset $\bbF_n^{\{4\}}$ is continuous, and vanishes identically unless the equality
\begin{equation} \label{norm1}
	\Vert J\Vert_1=j_0 + \ldots + j_3 = n-2
\end{equation}
holds. In particular, the volume cocycle $B_n:\bbF_n^4 \to \bbR$ is continuous on $\bbF_n^{\{4\}}$. 
\end{lem}

\begin{proof}
See discussion preceding \cite[Lem. 16]{BBI}. The restriction $B_n^J\vert_{\bbF_n^{\{4\}}}$  reduces to a hyperbolic volume function $\mathrm{Vol}_{\bbP}$ when \eqref{norm1} holds, and vanishes otherwise. This implies the continuity statement. Note that there are precisely $n(n^2-1)/6$ tuples $J$ that fulfill \eqref{norm1}. 
\end{proof}

\subsection{The Furstenberg boundary of $\SO_4(\bbC)$} \label{furstenberg}

The content of this subsection is classical, and can be found e.g. in \cite[\S I.7]{vdWaerden}. We continue to abbreviate $\bbP \coloneqq \bbP(\bbC^2)$.
 Observe first that $\SO_4(\bbC)$ acts on the variety of totally isotropic \emph{projective lines} in $\bbP(\bbC^4)$ (i.e. 2-dimensional linear subspaces of $\bbC^4$) with two orbits, $O^+$ and $O^-$. Respectively, these consist of elements 
\begin{align*}
	l^+_a &\coloneqq \big\{[1,-t,a,at]^\top \in \bbP(\bbC^4) \ \mid t \in \hat\bbC \big\} \in O^+, \\[-3pt]
	l^-_a &\coloneqq \big\{[1,a,-t,at]^\top \in \bbP(\bbC^4) \ \mid t \in \hat\bbC \big\} \in O^-,
\end{align*}
with $a \in \hat\bbC \cong \bbP$, given in coordinates with respect to the basis $\scB_2^{(+1,0)} = \{e_1,e_2,f_2,e_1\}$. These lines satisfy the properties \vspace{-4pt}
\begin{equation*} \begin{cases}
	l^+_a \cap l_{b}^+ \!\!\!\! &= l_a^- \cap l_{b}^- = \emptyset \quad \mbox{for } a \neq b, \\
	l^+_a \cap l_b^- \!\!\!\! &= [1,b,-a,ab]^\top.
\end{cases} \end{equation*}
The $\SO_4(\bbC)$-action on the product $\clF \coloneqq O^+ \times O^-$ is transitive, and the stabilizer of the pair $(l_0^+,l_0^-) \in \clF$ equals the subgroup of upper-triangular orthogonal matrices. Thus, we identify the Furstenberg boundary of $\SO_4(\bbC)$ with $\clF$. 

There exists an exceptional isogeny $\theta: \SL_2(\bbC) \times \SL_2(\bbC) \twoheadrightarrow \SO_4(\bbC)$ that descends to an isomorphism of Lebesgue $\SL_2(\bbC)\times \SL_2(\bbC)$-spaces, 
\begin{equation} \label{eq:partial_kappa_D}
	\partial\theta: \bbP \times \bbP \to \clF, \quad \partial\theta(a,b) \coloneqq (l_a^+,l_b^-), \ \ a,b \in \hat\bbC.
\end{equation}

The following embedding of the Furstenberg boundary $\clF$ into the full flag variety of $\bbC^4$ plays a key role in our proof of \autoref{lem_ind_base}. 

\begin{defn} \label{rho_F}
Let $\rho: \clF \to \bbF_4$ be the $\SO_4(\bbC)$-equivariant embedding defined as
\begin{align*}
	\rho(l^+,l^-) &\coloneqq \big(\{0\} \subset \ l^+ \cap l^- \ \subset \ l^+ \ \subset l^+ + l^- \ \subset \ \bbC^4\big).\label{xi_d}
\end{align*}
An analogous map can be obtained if we choose $l^-$ instead of $l^+$ as the 2-dimensional subspace. 
\end{defn}

\begin{rem} \label{rem:not_mcp}
The image of $\rho$ is a null set. In particular, $\rho$ does not preserve measure classes.  
\end{rem}

\subsection{Generic values of the volume cocycle of $\SO_4(\bbC)$}
We introduce a notion of genericity for tuples of points in the Furstenberg boundary $\clF$. 

\begin{defn}
For any $k \in \bbN$, we set
\begin{align*}
	\clF^{\{k+1\}} &\coloneqq \partial\theta^{\,\times(k+1)} \big(\bbP^{\{k+1\}} \times \bbP^{\{k+1\}}\big) \subset \clF^{k+1},
\end{align*}
where the brackets in the upper index of $\bbP$ refer to the generic subsets \eqref{Pk_brackets} in the classical setting. The product $\bbP^{\{k+1\}} \times \bbP^{\{k+1\}}$ is regarded as a subset of $(\bbP \times \bbP)^{k+1}$ in the obvious way. 
\end{defn}

\begin{lem} \label{SG transitive}
The following statements hold:
\begin{enumerate}[label=\emph{(\roman*)},leftmargin=25pt]
\item The set $\clF^{\{k+1\}}$ is an $\SO_4(\bbC)$-invariant, co-null Borel subset of $\clF^{k+1}$, and the equality $\partial_i(\clF^{\{k+1\}}) = \clF^{\{k\}}$ holds for every $i \in \{0,\ldots,k+1\}$. 
\item The action $\SO_4(\bbC)\curvearrowright \clF^{\{3\}}$ is transitive.
\item The image $\rho^{\, \times 4}(\clF^{\{4\}})$ is contained in the set $\bbF_4^{\{4\}}$ of 4-tuples in general position. 
\end{enumerate}
\end{lem}

\begin{proof}

Item (i) follows from the fact that $\partial\theta$ is an isomorphism of Lebesgue $\SL_2(\bbC) \times \SL_2(\bbC)$-spaces, and $\bbP^{\{k+1\}} \times \bbP^{\{k+1\}}$ is invariant, Borel and co-null in $(\bbP\times\bbP)^{k+1}$. Item (ii) is a consequence of the transitivity of the action $\SL_2(\bbC) \curvearrowright \bbP^{\{3\}}$. For (iii), we let $\mathbf{a}=(a_0,\ldots,a_3)$ and $\mathbf{b}=(b_0,\ldots,b_3) \in \bbP^{\{4\}}$. The genericity of these tuples means that their entries are pairwise distinct when regarded under the identification $\bbP\cong \hat\bbC$. Set $l_i^+ \coloneqq l_{a_i}^+$ and $l_i^- \coloneqq l_{b_i}^-$, so that \vspace{-3pt}
\begin{equation*}
\big((l_{0}^+,l_{0}^-),\ldots,(l_{3}^+,l_{3}^-)\big)= (\partial\theta_{0})^4(\mathbf{a},\mathbf{b}) \in \clF^4.
\end{equation*}
We claim that the image under $\rho^{\,\times 4}$ of the tuple $((l_{0}^+,l_{0}^-),\ldots,(l_{3}^+,l_{3}^-))$ is in general position. Indeed, if we set $p_i \coloneqq l_{i}^+ \cap l_{i}^- \in \clP_{2}$, then, for instance, for distinct $i,j \in \{0,1,2,3\}$, we have
\begin{align*}
	\dim(p_i + p_j) &= 1+1-\dim(l_{i}^+\cap l_{i}^-\cap l_{j}^+\cap l_{j}^-) = 2,\\
	\dim(p_i +l_{j}^+) &= 1+2-\dim(l_{i}^+\cap l_{i}^-\cap l_{j}^+) = 3, \\
	\dim(p_i +(l_{j}^+ + l_{j}^-)) &= 1 + 3 - \dim(l_{i}^+\cap l_{i}^-\cap(l_{j}^+ + l_{j}^-)) = 4,
\end{align*}
where we used that $l_i^+ \cap l_j^+ = l_i^- \cap l_j^- = \{0\}$ and
\[
	l_i^+ \cap l_i^- \cap(l_j^+ + l_j^-) \subset (l_i^+ \cap l_i^- \cap l_j^+) + (l_i^+ \cap l_i^- \cap l_j^-) = \{0\}.
\]
The remaining verifications are completely analogous to the three given above.
\end{proof}

In the following lemma, we examine the values of the volume cocycle $B_n$ over the image of the generic set $\clF^{\{4\}}$ via the inclusion $\rho^{\,\times 4}:\clF^4 \to \bbF_4^4$. For simplicity, in its proof we will write 4-tuples of indices $(j_0,\ldots,j_3)$ as $j_0\cdots j_3$, omitting parentheses and commas. 

\begin{lem} \label{non_zero}
For any $(a,b) \in \hat\bbC^2$, let $F_{(a,b)} \coloneqq \rho\big(\partial\theta(a,b)\big)$ and $F_a \coloneqq F_{(a,a)}$. Then for all $(a,b) \in (\bbC \smallsetminus \{0,1\})^2$, the equality 
\[
	B_4(F_\infty,F_0,F_1,F_{(a,b)}) = 2 \left(\clD(a) + \clD(b)\right)
\]
holds. In particular, the map 
$
	B_{4} \, \circ \, \rho^{\,\times 4}\vert_{\clF^{\{4\}}}: \clF^{\{4\}} \to \bbR
$
is non-zero. 
\end{lem}
\begin{proof}
We claim the following values for the non-vanishing summands $B_4^J$ of $B_4$
(see \autoref{gen_continuity}): 
\begin{equation} \label{terms_J}B_4^J(F_\infty,F_0,F_1,F_{(a,b)}) = \left\{\begin{array}{ll}
	\phantom{-}\clD(b) & \mbox{ if } J \in \{2000,0200,0020,0002\},\\
	-\clD(b/a) & \mbox{ if } J \in \{1100,0011\}, \\
	\phantom{-}\clD\left(\frac{1-b}{1-a}\right) & \mbox{ if } J \in \{1010,0101\}, \\[4pt]
	-\clD\left(\frac{a(1-b)}{b(1-a)}\right) & \mbox{ if } J \in \{1001,0110\}. \\
\end{array}\right.
\end{equation}
First, if we let $\phi_{ab} \coloneqq e_1+b\,e_2-a\,f_2+ab\,f_1$ for any $a,b \in \hat\bbC$, then
\begin{equation*} 
	l_a^+ = \langle e_1-a\,f_2, \ e_2+a\,f_1 \rangle, \quad l_b^- = \langle e_1+b\,e_2, \ b\,f_1-f_2 \rangle, \qand  l_a^+ \cap l_b^- = \langle \phi_{ab} \rangle, 
\end{equation*}
and therefore, for $a,b \in \bbC\smallsetminus \{0,1\}$,
\begin{align*}
	F_\infty &=\big(\{0\} \subset \langle f_1\rangle \subset \langle f_1,f_2 \rangle \subset \langle f_1,f_2,e_2 \rangle \subset \bbC^4\big), \\[-2pt]
	F_0 &=\big(\{0\} \subset \langle e_1\rangle \subset \langle e_1,e_2 \rangle \subset \langle e_1,e_2,f_2 \rangle \subset \bbC^4\big), \\[-2pt]
	F_{(a,b)} &=\big(\{0\} \subset \langle \phi_{ab} \rangle \subset \langle \phi_{ab}, e_1-a\,f_2 \rangle \subset \langle \phi_{ab},e_1-a\,f_2, e_1+b\,e_2 \rangle \subset \bbC^4\big).
\end{align*}
Now, we have:\vspace{-4pt}
\begin{align*}
B_4^{2000}(F_\infty,F_0,F_1,F_{(a,b)}) &= \mathrm{Vol} \ \bigg[\frac{\langle f_1,f_2,e_2 \rangle + \langle e_1 \rangle + \langle \phi_{11}\rangle + \langle \phi_{ab} \rangle}{\langle f_1,f_2\rangle + \{0\} + \{0\} + \{0\}};\big(e_2,e_1,\phi_{11},\phi_{ab}\big)\bigg] \\
	&=\mathrm{Vol} \ \big[V^{(+1,0)}_2/\langle f_1,f_2\rangle;\big(e_2,e_1,\phi_{11},\phi_{ab}\big)\big]\\
	&\overset{(\ast)}{=}\mathrm{Vol}\left[\bbC^2;\begin{pmatrix}
	0 & 1 & 1 & 1 \\[-3pt]
	1 & 0 & 1 & b
\end{pmatrix}\right] \\
	&= \mathrm{Vol}_{\bbP}(\infty,0,1,b)=\clD(b), \qand \\[-30pt]
\end{align*}
\begin{align*}
B_4^{0011}&(F_\infty,F_0,F_1,F_{(a,b)}) = \\
	&= \mathrm{Vol} \ \bigg[\frac{\langle f_1 \rangle + \langle e_1 \rangle + \langle \phi_{11},e_1\!-\!f_2 \rangle + \langle \phi_{ab},e_1\!-\!af_2 \rangle}{\{0\} + \{0\} + \langle \phi_{11}\rangle+ \langle \phi_{ab}\rangle};\big(f_1,e_1,e_1-f_2,e_1-af_2\big)\bigg] \\
	&=\mathrm{Vol} \ \big[V_2/ \langle \phi_{11}, \phi_{ab}\rangle;\big(f_1,e_1,e_1-f_2,e_1-af_2\big)\big] \\
	&\!\overset{(\ast\ast)}{=}\!\mathrm{Vol}\left[\bbC^2;\begin{pmatrix}
	0 & \frac{a-b}{a-1} & 1 & b \\[-2pt]
	-b^{-1} & 0 & 1 & a
\end{pmatrix}\right] \\[-2pt]
	&=\mathrm{Vol}_{\bbP}(\infty,0,1,a/b) = \clD(a/b) = -\clD(b/a).
\end{align*}
The last equality in the case $J=0011$ follows from one of the symmetries in \eqref{symmetries_bw}. The equalities $(\ast)$ and $(\ast\ast)$ follow from the notion of equivalence in $\sigma_3$, implemented respectively by the epimorphisms $T_1,T_2: V_2 \twoheadrightarrow \bbC^2$ defined as:
\begin{align*}
	T_1&: e_1 \mapsto \big(1, 0\big)^\top, \qquad\! e_2 \mapsto  \big(0, 1\big)^\top, \qquad f_2 \mapsto \big(0,0\big)^\top, \qquad\ f_1 \mapsto \big(0,0\big)^\top; \\[-4pt]
	T_2&: e_1 \mapsto \left({\textstyle \frac{\scriptstyle a-b}{\scriptstyle a-1}},0\right)^\top, \quad\!\!\! e_2 \mapsto \left(-1, {\textstyle\frac{\scriptstyle 1-b}{\scriptstyle b}} \right)^\top, \, f_2 \mapsto \left({\textstyle\frac{\scriptstyle 1-b}{\scriptstyle a-1}}, -1\right)^\top, \ f_1 \mapsto \left(0, -b^{-1}\right)^\top.
\end{align*}
Note that $\ker T_1 = \langle f_1,f_2 \rangle$ and $\ker T_2 = \langle \phi_{11}, \phi_{ab} \rangle$. To conclude, we add the terms in \eqref{terms_J} and use the fact that $\clD$ satisfies the Spence--Abel equation \eqref{spence_abel}: 
\begin{align*}
	B_4(F_\infty,F_0,F_1,F_{(a,b)}) &= 4 \, \clD(b) + 2 \left(-\clD\!\left({\textstyle \frac{a(1-b)}{b(1-a)}}\right)+\clD\!\left({\textstyle \frac{1-b}{1-a}}\right)-\clD\!\left({\textstyle \frac{b}a}\right)\right) \\
	&= 4 \, \clD(b)+2\big(\clD(a) - \clD(b)\big) = 2\big(\clD(a) + \clD(b)\big).\qedhere
\end{align*}
\end{proof}

\subsection{Proof of \autoref{lem_ind_base}} \label{subsec:ind_base}
If $\rho: \clF \to \bbF_4$ induced a morphism of $\Linfty$-complexes, then it would, by virtue of functoriality \cite[Prop. 8.4.2]{Monod-Book}, implement the map $\res_2$ from \autoref{lem_ind_base} at the cocycle level, and the injectivity of $\res_2$ would be a consequence of \autoref{non_zero}. However, since $\rho$ does not preserve measure classes, it cannot possibly induce a map of $\Linfty$-spaces.

To circumvent this measure-theoretical difficulty, we will consider the commutative diagram
\begin{equation} \label{eq:discretebc}
\begin{gathered}
\xymatrixcolsep{4pc}
\xymatrixrowsep{1.2pc}
\xymatrix{
\Hb^3(\SL_{4}(\bbC)) \ar[r]^-{\res_2} \ar[d]_{\id^\ast} & \Hb^3(\SO_{4}(\bbC)) \ar[d]^{\id^\ast}\\
\Hb^3(\SL_{4}(\bbC)^\delta) \ar[r]^-{\res_2^\delta} & \Hb^3(\SO_4(\bbC)^\delta)
}
\end{gathered} \vspace{-2pt}
\end{equation}
where the upper indices $(-)^\delta$ on groups indicate that they are being regarded as discrete. In words, the map $\res_2$ ``factors'' over its discrete analogue $\res_2^\delta$. This observation proves advantageous for us, as the latter map is more tractable than the former at the level of cocycles. 

As actions of discrete groups, $\SL_4(\bbC)^\delta \curvearrowright \bbF_4$ and $\SO_4(\bbC)^\delta \curvearrowright \clF$ are amenable. In fact, their point-stabilizers are isomorphic to upper-triangular subgroups, and hence abstractly solvable. Thus, we have canonical isomorphisms
\begin{equation*}
	\Hb^\ast(\SL_{4}(\bbC)^\delta) \cong \HH^\ast(\ell^\infty(\bbF_n^{\bullet+1})^{\SL_{4}(\bbC)}) \qand
	\Hb^\ast(\SO_{4}(\bbC)^\delta) \cong \HH^\ast(\ell^\infty(\clF^{\bullet+1})^{\SO_{4}(\bbC)}). 
\end{equation*}
Moreover, the now well-defined morphism of complexes\vspace{-3pt}
\begin{equation} \label{eq:rho_d}
	\rho^\ast: \ell^\infty(\bbF_4^{\bullet+1}) \to \ell^\infty(\clF^{\bullet+1}) \vspace{-3pt}
\end{equation}
implements the restriction $\res_2^\delta: \Hb^\ast(\SL_4(\bbC)^\delta) \to \Hb^\ast(\SO_4(\bbC)^\delta)$ (see \cite[Thms. 4.23, 4.15]{Frigerio}). 

Our last ingredient in the proof of \autoref{lem_ind_base} is an ``equivariant lifting'' theorem, due to Monod \cite{Monod-Lifting}, that will enable the realization at the level of cocycles of the maps $\id^\ast$ in the diagram \eqref{eq:discretebc}. For any measure space $X$, we denote the Banach space of bounded measurable functions $X \to \bbR$ (without identification modulo null sets) by $\Binfty(X)$. If $\varphi \in \Binfty(X)$, we write $[\varphi]$ for its equivalence class in $\Linfty(X)$. 

\begin{thm}[{see \cite{Monod-Lifting}}] \label{lifting}
Let $G$ be a locally compact group with an amenable $C^1$-action on a differentiable manifold $X$. Then for every $k \in \bbN$, there exists a $G$-equivariant lifting \vspace{-3pt}
\[
	s^k: \Linfty(X^{k+1}) \to \Binfty(X^{k+1})\vspace{-3pt}
\] 
of the canonical projection $\Binfty(X^{k+1}) \twoheadrightarrow \Linfty(X^{k+1})$ such that:
\begin{enumerate}[label=\emph{(\roman*)}]
 	\item $s^\bullet$ is a morphism of complexes, and 
	\item if $x \in X^{k+1}$ is a continuity point of a function $\varphi \in \Binfty(X^{k+1})$, then $(s[\varphi])(x) = \varphi(x)$. 
\end{enumerate}
\end{thm}

\begin{proof}[About the proof]
The existence of the $s^k$ is \cite[Thm. A]{Monod-Lifting}. For (i), see \cite[Rem. 1]{Monod-Lifting}. For (ii), see the comment at the end of the proof of \cite[Thm. B]{Monod-Lifting}. 
\end{proof}

Let us fix a morphism of complexes $s$ that results from the composition of an equivariant lifting as in \autoref{lifting} and the surjection $\Binfty \twoheadrightarrow \ell^\infty$: 
\begin{equation} \label{eq:s}
\begin{gathered}
\xymatrixcolsep{1pc}
\xymatrix{
	\Linfty(\bbF_4^{\bullet+1}) \ar[r] \ar@/_1.3pc/@{-->}[rr]_s& \Binfty(\bbF_4^{\bullet+1}) \ar@{->>}[r] & \ell^\infty(\bbF_4^{\bullet+1})
}
\end{gathered} \vspace{-6pt}
\end{equation}
Again by functoriality, the morphism $s$ implements $\id^\ast:\Hb^\ast(\SL_4(\bbC)) \to \Hb^\ast(\SL_4(\bbC)^\delta)$ at the level of cocycles. Observe that $s[B_{4}] = B_{4}$ on $\bbF_4^{\{4\}}$ by \autoref{gen_continuity}.

\begin{proof}[Proof of \autoref{lem_ind_base}]
Since $\Hb^3(\SL_4(\bbC)) \cong \bbR$, it suffices to prove that the composition $\res^\delta \circ \id^\ast$ in the diagram \eqref{eq:discretebc} is not the zero map. This will follow from exhibiting a non-trivial 3-cocycle in the image of $\rho^\ast \circ s$, for $\rho^\ast$ and $s$ as introduced in \eqref{eq:rho_d} and \eqref{eq:s}, respectively.

Let $[B_4]$ be the class up to null sets of the volume cocycle $B_4 \in \Binfty(\bbF_4^4)$, and set $\clB_2 \coloneqq (\rho^\ast \circ s)[B_4] \in \ell^\infty(\clF^{4})$. Observe that $\clB_2$ is an $\SO_4(\bbC)$-invariant bounded 3-cocycle. Moreover, 
\[
	\clB_2(x) = s[B_4]\big(\rho^{\,\times 4}(x)\big) = B_4\big(\rho^{\,\times 4}(x)\big) \quad \mbox{for all } x \in \clF^{\{4\}},
\]
in virtue of \autoref{gen_continuity} and \autoref{SG transitive} (iii). Thus, we deduce from \autoref{non_zero} that $\clB_2$ does not vanish identically on $\clF^{\{4\}}$, and this implies that $\clB_2$ is not a coboundary. Indeed, by \autoref{SG transitive} (ii), any $\SO_4(\bbC)$-invariant function $\varphi \in \ell^\infty(\clF^3)$ must take a constant value $c$ on $\clF^{\{3\}}$, and therefore, for every $x \in \clF^{\{4\}}$, any coboundary $\dd\!\varphi$ must satisfy the equality
\[
	\dd\!\varphi(x) = {\textstyle \sum_{i=0}^3} (-1)^i \, \varphi(\partial_i(x)) = c - c + c - c = 0. \qedhere
\]
\end{proof}

We finish this subsection by computing the norm of the restricted class $\res_2(b_4) \in \Hb^3(\SO_4(\bbC))$. In particular, this proves that the restrictions $\res_r$ need not be isometric in the $\mathbf{D}_r$ case. 

\begin{prop} \label{norm_low}
The restricted class $\res_2(b_4) \in \Hb^3(\SO_4(\bbC))$ has Gromov norm equal to $4\,v_{\bbH^3}$, and the map
$
	\res_2: \Hb^3(\SL_4(\bbC)) \hookrightarrow \Hb^3(\SO_4(\bbC))
$
has operator norm equal to $2/5$. 
\end{prop}

\begin{proof}
Let $\clB_2 \in \ell^\infty(\clF^4)$ be the cocycle defined in the proof of \autoref{lem_ind_base}. Since $\clB_2$ is the restriction of a Borel cocycle, its class $[\clB_2] \in \Linfty(\clF^4)$ is well-defined. Now, as in \eqref{eq:s}, let \vspace{-3pt}
\[
	s': \Linfty(\clF^{\bullet+1}) \to \ell^\infty(\clF^{\bullet+1}) \vspace{-3pt}
\]
be the $\SO_4(\bbC)$-equivariant morphism of complexes defined by the composition of a lifting as in \autoref{lifting} and the surjection $\Binfty \twoheadrightarrow \ell^\infty$. 
Moreover, the equality $s'[\clB_2] = \clB_2$ holds since $\clB_2$ is continuous on the co-null set $\clF^{\{4\}} \subset \clF^4$. By \cite[Thm. 7.5.3]{Monod-Book} and the commutativity of the diagram \eqref{eq:discretebc}, we deduce that the norm of $\res_2(b_4)$ must equal the norm of the cohomology class defined by $[\clB_2]$. 

Now, by the essential transitivity of the action $\SO_4(\bbC) \curvearrowright \clF^3$ (see \autoref{SG transitive} (ii)), there are no non-trivial invariant 3-coboundaries in $\Linfty(\clF^{\bullet+1})$. Hence,
$
	\Vert \res_2(b_4) \Vert = \Vert [\clB_2] \Vert_\infty = 4 \cdot v_3
$
by \autoref{non_zero}. For the operator norm of $\res_2$, note simply that $\Vert b_4 \Vert = 10 \cdot v_3$. 
\end{proof}

\subsection{End of proof of Theorems \ref{new_stability} and \ref{restriction_inj}}
For any $r \geq r_1 = 2$, consider the commutative cube
\begin{equation} \label{cube}
\hspace{-50pt}\begin{gathered}
	\xymatrixcolsep{0.3pc}
	\xymatrixrowsep{0.3pc}
  	\xymatrix{
& \Hb^3(G_{r+1}) \ar@{^{(}->}[ld] \ar@{^{(}->}'[d][dd] && \makebox[\widthof{$B$}][l]{$\Hb^3(\GL_{n(r+1)}(\bbC))$} \ar[dd]_-\cong \ar[ld]_\cong \ar[ll] \\
\Hb^3(S_{r+1}) \ar[dd] && \Hb^3(\SL_{n(r+1)}(\bbC)) \ar[ll]_(.41){\res_{r+1}} \ar[dd]_(.30)\cong & \\
& \Hb^3(G_r) \ar@{^{(}->}[ld] && \makebox[\widthof{$B$}][l]{$\Hb^3(\GL_{n(r)}(\bbC))$} \ar'[l][ll] \ar[ld]_{\cong}\\ 
\Hb^3(S_r) && \Hb^3(\SL_{n(r)}(\bbC))\ar[ll]^-{\res_r}  &
	}
\end{gathered}
\end{equation}
where the horizontal arrows correspond to the restriction maps \vspace{-2pt}
\[
	\res_s: \Hb^3(\SL_{n(s)}(\bbC)) \to \Hb^3(S_s) \qand \Hb^3(\GL_{n(s)}(\bbC)) \to \Hb^3(G_s),\vspace{-2pt}
\]
the arrows pointing frontwards to the restriction maps\vspace{-2pt}
\[
	\sigma_{n(s)}: \Hb^3(\GL_{n(s)}(\bbC)) \to \Hb^3(\SL_{n(s)}(\bbC)) \qand \tau_s: \Hb^3(G_s) \to \Hb^3(S_s) \vspace{-2pt}
\]
for $s \in \{r,r+1\}$, and the vertical arrows to the maps induced by the block inclusions
\begin{align*}
	\iota_r: S_r \hookrightarrow S_{r+1}, \quad & \quad j_r: \SL_{n(r)}(\bbC) \hookrightarrow \SL_{n(r+1)}(\bbC), \\[-3pt]
	\iota_r: \, G_r \hookrightarrow G_{r+1},\quad &  \quad j_r: \GL_{n(r)}(\bbC) \hookrightarrow \GL_{n(r+1)}(\bbC),
\end{align*}
Observe that $\sigma_{n(s)}$ is an isomorphism by \cite[Cor. 8.5.5]{Monod-Book} and that $\tau_s$ is an injection by \cite[Prop. 8.6.2]{Monod-Book} for both $s \in \{r,r+1\}$, since $S_s < G_s$ is a finite-index normal subgroup. Moreover, the maps $j_r^\ast$ induced in degree three by $j_r$, both for $\GL$ and $\SL$, are isomorphisms due to \autoref{new_stability} for the $\mathbf{A}_r$ case. By \hyperlink{bootstrap_bis}{Key Lemma \emph{bis}}, the map $\iota_r^\ast:\Hb^3(G_{r+1}) \to \Hb^3(G_r)$ is an injection. 

The next lemma will serve as the induction step in the proofs of Theorems \hyperlink{new_stability_bis}{\ref{new_stability} \emph{bis}} and \hyperlink{restriction_inj_bis}{\ref{restriction_inj} \emph{bis}}. 

\begin{lem} \label{ind_step}
If, for any $r \geq 1$, the map $\res_r: \Hb^3(\SL_{n(r)}(\bbC)) \to \Hb^3(S_r)$
is an isomorphism, then all the arrows in the diagram \eqref{cube} are isomorphisms.
\end{lem}   

\begin{proof}
Note, by examining the bottom face of the cube, that if $\res_r$ is an isomorphism, then $\tau_r$ is surjective. This, in turn, implies that all the arrows in the bottom face are isomorphisms. Relying on these isomorphisms, a similar analysis of, first, the back face; then, the left face; and finally, the front face of \eqref{cube}, complete the proof of the lemma. 
\end{proof}

\begin{proof}[Proof of \hyperlink{restriction_inj_bis}{Theorem \ref{restriction_inj} \emph{bis}} for $\mathbf{D}_r$]
We prove by induction that $\res_r: \Hb^3(\SL_{n(r)}(\bbC)) \to \Hb^3(S_r)$ is an isomorphism for every $r \geq r_0 = 3$. By \autoref{ind_step}, if $\res_r$ is an isomorphism for any $r \geq 3$, then so is $\res_{r+1}$. Thus, the claim follows once we establish that 
\[
	\res_{r_0} = \res_3: \Hb^3(\SL_6(\bbC)) \to \Hb^3(\SO_6(\bbC))
\]
is an isomorphism. The group $\SO_6(\bbC)$ admits an exceptional isogeny $\SL_4(\bbC) \twoheadrightarrow \SO_6(\bbC)$, as in e.g. \cite{vdWaerden,Tao}, which induces an isomorphism in bounded cohomology \cite[Cor. 8.5.2]{Monod-Book}. After applying \autoref{new_stability} for $\mathbf{A}_r$, we deduce the chain of isomorphisms
\begin{align*}
	\Hb^3(\SO_6(\bbC)) &\cong \Hb^3(\SL_4(\bbC)) \cong \Hb^3(\SL_6(\bbC)) \cong \bbR. 
\end{align*}
Thus, it suffices to prove that $\res_3$ is not identically zero. 
For this purpose, consider the front face of the commutative cube \eqref{cube}, with $r=2$. Since the bottom arrow $\res_2$ is injective by \autoref{lem_ind_base}, we deduce our claim. 
\end{proof}

\begin{proof}[Proof of \hyperlink{new_stability}{Theorem \ref{new_stability} \emph{bis}} for $\mathbf{D}_r$]
By \autoref{ind_step}, if $\res_r$ is an isomorphism for any $r \geq r_0=3$, then so is the map $\iota_r^\ast:\Hb^3(S_{r+1}) \to \Hb^3(S_r)$. \hyperlink{restriction_inj_bis}{Theorem \ref{restriction_inj} \emph{bis}} finishes the proof. 
\end{proof}

\begin{rem}
We note that $\res_2: \Hb^3(\SL_4(\bbC))\hookrightarrow \Hb^3(\SO_4(\bbC))$ is not an isomorphism, since $\Hb^3(\SL_4(\bbC)) \cong \bbR$, and $\SO_4(\bbC)$ is isogenous to $\SL_2(\bbC) \times \SL_2(\bbC)$, which has 2-dimensional $\Hb^3$.
\end{rem}

\section{Proof of Theorem 1}
Let $(\varepsilon,d) \in \{(+1,1),(-1,0),(+1,0)\}$. We will prove \hyperlink{main_thm_bis}{Theorem \ref{main_thm} \emph{bis}}, that is, that the comparison map $c^3:\Hb^3(S_r) \to \HH^3(S_r)$ is an isomorphism for all $r \geq r_0$, by induction on $r$. 

The base case $r=r_0$ holds due to the exceptional Lie algebra isomorphisms in low rank, which place $S_{r_0}$ inside the $\mathbf{A}_r$ family. For the induction step, consider the commutative diagram
\begin{equation*} 
\begin{gathered}
\xymatrixcolsep{5pc}
\xymatrixrowsep{1.2pc}
\xymatrix{\Hb^3(S_{r+1}) \ar[d]_{\iota_r^\ast}^\cong \ar[r]^-{c^3} & \HH^3(S_{r+1}) \ar[d]^{\iota_r^\ast}_\cong \\
\Hb^3(S_{r}) \ar[r]_{c^3}^\cong & \HH^3(S_r)}
\end{gathered}
\end{equation*}
The induction hypothesis is the fact that the bottom arrow an isomorphism. That the left arrow is an isomorphism is \hyperlink{new_stability}{Theorem \ref{new_stability}}. That the right one is an isomorphism is the stability of classical families in continuous cohomology, for which we argue as follows. For any non-compact semisimple Lie group $G$, the continuous cohomology $\HH^\ast(G)$ is isomorphic to the cohomology of its compact symmetric space (see e.g. \cite[\S 5]{Stas}). In the cases covered by \autoref{main_thm}, such compact symmetric spaces are diffeomorphic to compact classical groups. The stability of the (space) cohomology of those groups is recorded in \cite[Thm. III.6.5]{Toda-Mimura}. \qed

\section{About the Gromov norm on $\Hb^3$ of complex classical groups} \label{subsec:norm_D}

For the purpose of this final section, let us adapt the notation introduced in Section \ref{sec:notation} to cover also the complex family $\mathbf{A}_r$, as follows. For $\scF \in \{\mathbf{A,\,B,\,C,\,D}\}$, let $S_{\scF_r}$ denote, respectively, the group 
\[
	\SL_{r+1}(\bbC), \quad \SO_{2r+1}(\bbC),\quad \Sp_{2r}(\bbC),\quad \SO_{2r}(\bbC),
\]
in the complex family $\scF_r$. As explained in Subsection \ref{subsec:assumptions}, the choice of $\scF \in \{\mathbf{B,\,C,\,D}\}$ corresponds, respectively, to choosing $(\varepsilon,d) \in \{(+1,1),(-1,0),(+1,0)\}$, and we will assume the notation introduced in Section \ref{sec:notation} for those cases. 

Let $b_{\scF_r} \in \Hb^3(S_{\scF_r})$ be the continuous bounded cohomology class defined as
\begin{equation*}
	b_{\scF_r} \coloneqq \begin{cases}
		b_{r+1} & \mbox{ if } \clF= \mathbf{A}, \\[-2pt]
		\frac12 \, \res_{r}(b_{2r+1}) & \mbox{ if } \clF = \mathbf{B}, \\
		\res_{r}(b_{2r}) & \mbox{ if } \clF = \mathbf{C}, \\
		\frac12 \, \res_{r}(b_{2r}) & \mbox{ if } \clF = \mathbf{D}, \\
	\end{cases}
\end{equation*} 
where $b_n \in \Hb^3(\SL_n(\bbC))$ is the bounded Borel class, and 
$
	\res_r: \Hb^3(\SL_{n(r)}(\bbC)) \to \Hb^3(S_{\scF_r})
$
is the restriction map. 

The next statement is an immediate corollary of Bucher--Burger--Iozzi \cite[Thm. 2]{BBI} (see \autoref{thm:irrep}), which gives the computation of the Gromov norm of the bounded Borel class, and of \autoref{thm:norms}.
\begin{cor} \label{Gromov_norms}
For $\scF \in \{\mathbf{A,\,B,\,C}\}$ and $r \geq 1$, the Gromov norm of $b_{\scF_r} \in \Hb^3(S_{\scF_r})$ is
\[
	\Vert b_{\scF_r} \Vert = \clI(\scF_r) \cdot v_{\bbH^3}, 
\]
where $v_{\bbH^3}$ is the maximal volume of an ideal tetrahedron in $\bbH^3$, and the number $\clI(\scF_r) \in \bbZ$ is given in the table \vspace{7pt}
\begin{equation*}
\begin{array}{|c||c|c|c|c|} \hline
	\scF_r & \mathbf{A}_r & \mathbf{B}_r & \mathbf{C}_r  \\ \hline
	\clI(\scF_r) & \frac{r(r+1)(r+2)}{6} & \frac{r(r+1)(2r+1)}3 & \frac{r(4r^2-1)}3 \\ \hline
\end{array}
\end{equation*}
\end{cor}

It was indicated to us by Marc Burger that the number $\clI(\scF_r)$ equals the Dynkin index of the principal $\SL_2(\bbC)$-homomorphism $\SL_2(\bbC) \to S_{\clF_r}$. We refer the reader to \cite{Oni} for the definition of Dynkin index, and \cite{Kos} for the theory of principal $\SL_2(\bbC)$-homomorphisms. The Dynkin index of the principal $\SL_2(\bbC)$-homomorphisms has been computed for all complex simple Lie groups \cite{Panyushev}, including exceptional ones. For $\clF=\mathbf{D}$, it is given by the formula
\[
	\clI(\mathbf{D}_r) = \frac{r(r-1)(2r-1)}3.
\]

Relying on the equality $\Vert b_{\mathbf{D}_2} \Vert = 2 \, v_{\bbH^3} = \clI(\mathbf{D}_2) \, v_{\bbH^3}$ established in \autoref{norm_low}, we conjecture that \autoref{Gromov_norms} holds also for $\scF=\mathbf{D}$.

\appendix
\section{Proofs of cross-ratio parametrization statements} \label{app:computation}

This appendix is devoted to the proof of the parametrizations via $\omega$-cross-ratios of the configuration spaces of four and five points in the variety $\clP_r$ of isotropic points. These were stated in the main text as \autoref{param4} and \autoref{param5}, respectively. We continue to adhere to the notation and standing assumptions in Subsections \ref{subsec:notation} and \ref{subsec:assumptions}. 

\subsection{Proof of \autoref{param4}}
Let us assume that $r \geq 2$. We consider the Borel functions $\Gamma,\,\Delta,\,\Delta^{1/2}, \, \varphi_-, \, \varphi_+:\bbC^2 \to \bbC$, defined as
\begin{align} 
	\Gamma(z_1,z_2) &\coloneqq 1 - z_1 - z_2, \label{def_gamma} \\
	\Delta(z_1,z_2) &\coloneqq \Gamma(z_1,z_2)^2-2(1+\varepsilon) \cdot z_1z_2, \label{def_delta}\\[-3pt]
	\Delta^{1/2}(z_1,z_2)  &\coloneqq \begin{cases}
		\Gamma(z_1,z_2) & \mbox{ if } \varepsilon = -1, \\
		\sqrt{\Delta(z_1,z_2)} & \mbox{ if } \varepsilon = +1, 
	\end{cases} \label{def_sqrt_delta} \\
	\varphi_{\eta}(z_1,z_2) &\coloneqq \left(\frac{\Delta^{1/2} \ + \ \eta\,\Gamma}2\right)(z_1,z_2), \quad \eta \in \{\pm 1\}, \label{def_varphi}
\end{align}
where $\sqrt{-}:\bbC \to \bbC$ is a Borel choice of complex square root. Abusively, we will use the same symbols to denote the Borel functions $\clP_r^4 \to \bbC$ defined a.e. by precomposing with $\pi_3 = (\CR_1,\CR_2)$. 

\begin{defn}
We let $\clP_r^{\{\!\{4\}\!\}}$ be the subset of $\clP_r^{\{4\}}$ consisting of tuples $(p_0,p_1,p_2,p_3)$ such that the subspace $\langle p_0,p_1,p_2,p_3 \rangle < V_r$ is non-degenerate (see \eqref{Pk_brackets} for the definition of $\clP_r^{\{4\}}$) . 
\end{defn}

\begin{lem} \label{delta_nondeg}
A tuple $[v_0,v_1,v_2,v_3] \in \clP_r^{(4)}$ belongs to $\clP_r^{\{\!\{4\}\!\}}$ if and only if the associated Gram matrix $(\omega(v_i,v_j))_{i,j}$
is non-singular. Moreover,
\begin{equation*} 
	\clP_r^{\{\!\{4\}\!\}} =\{(p_0,p_1,p_2,p_3) \in \clP_r^{(4)}\mid \Delta(p_0,p_1,p_2,p_3) \neq 0\}, 
\end{equation*}
and in particular, $\clP_r^{\{\!\{4\}\!\}}$ is a $G_r$-invariant, co-null Borel subset of $\clP_r^4$. 
\end{lem}

\begin{proof}
The matrix $(\omega(v_i,v_j))_{i,j}$ is singular if and only if there exists a non-trivial linear combination $v= \sum_i \alpha_i v_i$ such that $\omega(v_k,v) = 0$ for every $k \in \{0,\ldots,3\}$. This means precisely that either $v=0$ and the vectors $v_0,\ldots,v_3$ are linearly dependent, or $\langle v_0,\ldots,v_3\rangle$ is degenerate. The second part of the lemma follows from the fact that the Gram determinant of $\{v_0,v_1,v_2,v_3\}$ equals the quotient 
$
\Delta[v_0,v_1,v_2,v_3]\,/(\omega(v_0,v_1)\,\omega(v_2,v_3))^2. 
$
\end{proof}

The subspace $L<V_r$ generated by any four lines $(p_0,p_1,p_2,p_3) \in \clP_r^{\{\!\{4\}\!\}}$ is isomorphic to the sum $\clH \oplus \clH$. In \autoref{key lemma} below, we will construct an adapted basis of $L$ in terms of any such points $p_0,p_1,p_2,p_3$. For that purpose, we introduce perpendicular projections with respect to $\omega$ onto hyperbolic planes, and list relevant properties.

\begin{defn}
Let $\clH_0 < V_r$ be a hyperbolic plane and $\{v_0,v_1\}$ a basis of isotropic vectors of $\clH_0$. The \emph{perpendicular projection} $\mbox{proj}=\mbox{proj}_{\clH_0}\colon V_r \twoheadrightarrow \clH_0$ onto $\clH_0$ is given by
\begin{equation*}
	\mbox{proj}(v) \coloneqq {\displaystyle\frac{\omega(v_1,v)}{\omega(v_1,v_0)} \, v_0 + \frac{\omega(v_0,v)}{\omega(v_0,v_1)}\, v_1,} 
\end{equation*}
with kernel $\clH_0^\perp$. We also let 
$\widehat{\,\cdot\,} \colon V_r \to \clH_0^\perp$ be the \emph{complement map}, defined as $\widehat{v} \coloneqq v-\mbox{proj}(v)$.
\end{defn}
 
\begin{lem} \label{lemma_hats}
Let $[v_0,v_1,v_2,v_3] \in \clP_r^{(4)}$, let $\widehat{\,\cdot\,}$ denote the complement map to $\langle v_0,v_1 \rangle$, let $v,w \in V_r$ be any two vectors, and let $u \in V_r$ be isotropic. Then the following identities hold:
	\begin{enumerate}[label=\emph{(\roman*)},itemsep=1pt]
		\item $\omega(\widehat{v},\widehat{w}) = \omega(\widehat{v},w) = \omega(v,\widehat{w})$.
		\item $q(\widehat{u}) = -(1+\varepsilon)\, \big( \omega(v_1,u) \, \omega(u,v_0)/\omega(v_0,v_1) \big)$.
		\item $\omega(\widehat{v_2},\widehat{v_3})= \omega(v_2,v_3) \cdot \Gamma[v_0,v_1,v_2,v_3]$. 
		\item $\omega(\widehat{v_2},\widehat{v_3})^2 - q(\widehat{v_2}) \, q(\widehat{v_3}) = \omega(v_2,v_3)^2 \cdot \Delta[v_0,v_1,v_2,v_3]$.
	\end{enumerate}
\end{lem}
\begin{proof}
Since $\omega(\widehat{v},\mbox{proj}(w)) = 0$, we have $\omega(\widehat{v},\widehat{w})=\omega(\widehat{v},w) - \omega(\widehat{v},\mbox{proj}(w)) =\omega(\widehat{v},w)$, the first equality in (i). The second one is analogous. Items (ii) and (iii) follow immediately from (i) and the definition of the complement map. For $\varepsilon = -1$, the identity (iv) holds since $\Delta = \Gamma^2$ and $q(\widehat{v_2}) \, q(\widehat{v_3}) = 0$. If $\varepsilon = +1$, then (iv) follows from (iii) and the next computation:
\begin{align*}	
	q(\widehat{v_2}) \, q(\widehat{v_3}) &= 4 \cdot  \frac{\omega(v_1,v_2) \, \omega(v_2,v_0)}{\omega(v_0,v_1)} \cdot \frac{\omega(v_1,v_3) \, \omega(v_3,v_0)}{\omega(v_0,v_1)} \\
	&= 4 \cdot \omega(v_2,v_3)^2 \cdot \frac{\omega(v_1,v_3) \, \omega(v_2,v_0)}{\omega(v_1,v_0) \, \omega(v_2,v_3)} \cdot \frac{\omega(v_2,v_1) \, \omega(v_0,v_3)}{\omega(v_2,v_3) \, \omega(v_0,v_1)} \\
	&= 4 \cdot \omega(v_2,v_3)^2 \cdot (\CR_1^{-1} \cdot \CR_2)[v_0,v_1,v_2,v_3]. \qedhere
\end{align*}
\end{proof}

The following technical lemma plays a key role in the proof of \autoref{param4}. 

\begin{lem} \label{key lemma}
Let $\mathbf{p} = [v_0,v_1,v_2,v_3] \in \clP_r^{\{\!\{4\}\!\}}$, let $\clH_0 \coloneqq \langle v_0,v_1 \rangle$ be the hyperbolic plane generated by the first two vectors, and let $\widehat{\,\cdot\,}$ denote the complement map to $\clH_0$. We also let the quantities $\Pi=\Pi(v_0,v_1,v_2)$, $\lambda=\lambda(v_0,v_1,v_2)$, $\mu = \mu(v_0,v_1,v_2,v_3)$ be defined as
\begin{align} \label{lambda,pi}
	\Pi &\coloneqq \omega(v_0,v_1) \, \omega(v_1,v_2) \, \omega(v_2,v_0), \\ 
	\lambda &\coloneqq \frac{\sqrt{\Pi}}{\omega(v_1,v_0)} = \frac{\omega(v_1,v_2) \, \omega(v_0,v_2)}{\sqrt{\Pi}}, \\
	\label{mu}
\mu &\coloneqq \lambda^{-1} \cdot \omega(v_3,v_2).
\end{align}
Then $
	\clH_0' \coloneqq \clH_0^\perp \cap \langle v_0,v_1,v_2,v_3 \rangle = \langle \widehat{v_2},\widehat{v_3} \rangle  
$
is a hyperbolic plane, and the unique solutions $e', f'\in \clH_0'$ of the next linear systems form an adapted basis of $\clH_0'$:
\begin{equation} \label{systems}
	\left\{ \begin{array}{rl}
		\omega(e',v_2) &= \lambda \\
		\omega(e',v_3) &= \mu \,\varphi_-(\mathbf{p})
	\end{array} \right\} \qand
	\left\{ \begin{array}{rl} 
		\omega(f',v_2) &= -\lambda, \\
		\omega(f',v_3) &= \varepsilon \mu \,\varphi_+(\mathbf{p}),
	\end{array} \right\}
\end{equation}

\end{lem}

\begin{proof}
After writing $e'$ and $f'$ as linear combinations of the basis $\{\widehat{v_2},\widehat{v_3}\}$ of $\clH'$, one obtains that the coefficient matrix of both systems in \eqref{systems} has determinant
\[
	q(\widehat{v_2}) \, q(\widehat{v_3}) - \omega(\widehat{v_2},\widehat{v_3})^2 = - \omega(v_2,v_3)^2 \cdot \Delta(\mathbf{p}).
\]
This quantity is non-zero in virtue of \autoref{delta_nondeg}. Their explicit solutions are:
\begin{align} 
	\mbox{For } \varepsilon = -1: & \ \ \left\{\!\!\begin{array}{rl}
	e' &= (\lambda/\omega(\widehat{v_3},\widehat{v_2}))\cdot\widehat{v_3}, \\
	f' &= (-\mu /\omega(v_2,v_3)) \cdot \widehat{v_2} - (\lambda/\omega(\widehat{v_3},\widehat{v_2})) \cdot \widehat{v_3}.
\end{array} \right. \label{ef_-1} \\
	\mbox{For } \varepsilon = +1: & \ \ \left\{\!\!\begin{array}{rl}
	e' &= \phantom{-} \lambda/q(\widehat{v_2}) \cdot (\widehat{v_2} - r),  \\
	f' &= - \lambda/q(\widehat{v_2}) \cdot (\widehat{v_2} + r),
\end{array} \!\!\right. \label{ef_+1}
\end{align}
where 
\[
	r \coloneqq \frac{\omega(\widehat{v_2},\widehat{v_3}) \cdot \widehat{v_2}-q(\widehat{v_2}) \cdot \widehat{v_3}}{\omega(v_2,v_3) \cdot \Delta^{1/2}(\mathbf{p})} \in \clH'. \vspace{4pt} 
\]
The next identities are consequences of \autoref{lemma_hats} and facilitate the computations that verify that $e'$ and $f'$ as above are actually solutions of \eqref{systems} and an adapted basis of $\clH'_0$:
\begin{align*}
	(\varepsilon = -1) \quad& \omega(\widehat{v_2},v_2) = q(\widehat{v_2}) = 0, \quad \omega(\widehat{v_3},v_3) = q(\widehat{v_3}) = 0, \\
	 &\omega(\widehat{v_2},v_3) = \omega(v_2,\widehat{v_3}) = \omega(\widehat{v_2},\widehat{v_3}) = \omega(v_2,v_3) \cdot \Gamma(\mathbf{p}) \\[-3pt]
	 &\qquad \quad \ \ = \omega(v_2,v_3) \cdot \Delta^{1/2}(\mathbf{p}) \neq 0. \\[5pt]
	 (\varepsilon = +1) \quad& \omega(r,v_2) = \omega(r,\widehat{v_2}) = 0, \quad q(r) = -q(\widehat{v_2}) = 2\lambda^2 \neq 0,  \\
	 &\omega(r,v_3) = \omega(r,\widehat{v_3}) = \omega(v_2,v_3) \cdot \Delta^{1/2}(\mathbf{p}) \neq 0. \qedhere
\end{align*}
\end{proof}


\begin{proof}[Proof of \autoref{param4}]
Let $r \geq 2$. The theorem is proven once we establish the existence of a co-null subset $\Omega_3 \subset \bbC^2$ and a Borel map
$
	\Phi_3: \Omega_3 \to \clP_r^{4}
$
with the next three properties:
\begin{enumerate}[label=(\roman*),leftmargin=25pt]
	\item $\pi_3(\clP_r^{\{\!\{4\}\!\}}) \subset \Omega_3$.
	\item $\Phi_3(\Omega_3) \subset \clP_r^{\{\!\{4\}\!\}}$ and $\pi_3 \, \circ \, \Phi_3 = \id$.
	\item For all $\mathbf{p} \in \clP_r^{\{\!\{4\}\!\}}$, there exists $g \in G_r$ such that $\Phi_3 \, \circ \, \pi_3(\mathbf{p}) = g\, \mathbf{p}$.
\end{enumerate}

We proceed now to their construction. Let us define 
\begin{equation} \label{omega3}
	\Omega_3 \coloneqq \left\{(a_1,a_2) \in (\bbC^\times)^2\mid \  \Delta(a_1,a_2) \neq 0 \right\}.
\end{equation}
Note that $\Omega_3$ is a non-empty Zariski-open subset of $\bbC^2$, hence co-null, and condition (i) follows immediately from \autoref{delta_nondeg}. We define the isotropic vector $\phi_2$ and the Borel map $\phi_3 \colon \Omega_3 \to \langle e_r,e_{r-1},f_{r-1},f_r \rangle$ by
\begin{align*}
	\phi_2 &\coloneqq e_r + e_{r-1} - f_{r-1} + f_r, \\
	\phi_3(a_1,a_2) &\coloneqq a_1 \ e_r + \varphi_-(a_1,a_2) \ e_{r-1} + \varepsilon \, \varphi_+(a_1,a_2) \ f_{r-1} + \varepsilon \, a_2 \ f_r.
\end{align*}
The vector $\phi_3(a_1,a_2)$ is isotropic for any $(a_1,a_2) \in \Omega_3$. 
%
We let now $\Phi_3: \Omega_3 \to \clP_r^4$ be the map
\begin{equation} \label{bigphi3}
	 \Phi_3(a_1,a_2) \coloneqq \left[e_r,\,f_r,\,\phi_2,\, \ \phi_3(a_1,a_2)\right].
\end{equation}
The first three points in the triple are precisely the representative generic 3-tuple used in the proof of \autoref{3transitive}. After abbreviating $\phi_3 \equiv \phi_3(a_1,a_2)$, the identities given below hold true, implying that none of those quantities is zero and, hence, that $\Phi_3$ ranges onto $\clP_r^{(4)}$.
\begin{equation} \label{omega_pairing}
\begin{array}{lllllll}
	\omega(e_r,f_r) = 1, && \omega(e_r,\phi_2) = 1, && \omega(e_r,\phi_3) = \varepsilon a_2, \\
	\omega(f_r,\phi_2) = \varepsilon, &&
	\omega(f_r,\phi_3) = \varepsilon a_1, &&
	\omega(\phi_2,\phi_3) = \varepsilon. 
\end{array} \end{equation} 

From \eqref{omega_pairing}, we also derive that $\CR_i\big(\Phi_3(a_1,a_2)\big) = a_i$ for $i = 1,2$. Together with \autoref{delta_nondeg}, this implies that $\Phi_3(a_1,a_2)$ is in $\clP_r^{\{\!\{4\}\!\}}$ and that $\pi_4 \, \circ \, \Phi_3(a_1,a_2) = (a_1,a_2)$, establishing (ii). 

For (iii), we fix an arbitrary $\mathbf{p} = [v_0,v_1,v_2,v_3] \in \clP_r^{\{\!\{4\}\!\}}$. As in \autoref{key lemma}, we let $\lambda, \Pi, \mu$ be the quantities defined in function of $(v_0,v_1,v_2,v_3)$ from \eqref{lambda,pi}--\eqref{mu}, and $e',\,f'$ be the solutions of the systems \eqref{systems}. Then we set
\begin{equation} \label{4-basis}
	e'_r \coloneqq \frac{\lambda \, v_1}{\omega(v_1,v_2)}, \quad e'_{r-1} \coloneqq e', \quad f'_{r-1} \coloneqq f', \quad f'_r \coloneqq \frac{\omega(v_1,v_2)\,v_0}{\sqrt{\Pi}}.
\end{equation}
Observe that $\{e'_r,\,f'_r\}$ is an adapted basis of the hyperbolic plane $\clH_r \coloneqq \langle v_0,v_1 \rangle < V_r$, and that $\{e'_{r-1},f'_{r-1}\}$ is an adapted basis of $\clH_{r-1} \coloneqq \clH_r^\perp \cap \langle v_0,v_1,v_2,v_3 \rangle$. By Witt's lemma, we extend $\{e'_r,e'_{r-1},f'_{r-1},f'_r\}$ to a full adapted basis $\scB_r'$ of $V_r$. The automorphism $T \in \GL(V_r)$ that maps the standard adapted basis $\scB_r$ to $\scB'_r$ lies in $G_r$, and therefore, so does 
\begin{equation} \label{defn_g}
g = g(v_0,v_1,v_2,v_3) \coloneqq T^\top J_r.
\end{equation} 
Note that for every $i \in \{0,\ldots,3\}$, we have
\begin{align*}
	g \cdot v_i 
	&= \omega(e'_r,v_i)\, e_r + \omega(e'_{r-1},v_i)\, e_{r-1} + \omega(f'_{r-1},v_i)\, f_{r-1} + \omega(f'_r,v_i)\, f_r.
\end{align*}
Writing in coordinates with respect to the standard adapted basis $\scB_r$, we have the equalities below, where the second one follows after multiplying the fourth vector by $\mu^{-1}$:
\begin{equation*}
	g \cdot \mathbf{p} 
		= \left[ \begin{array}{cccc}
			1 & 0 & \lambda & \lambda \, \omega(v_1,v_3)/\omega(v_1,v_2) \\
			0 & 0 & \lambda & \mu \, \varphi_-(\mathbf{p}) \\
			\mathbf{0} & \mathbf{0} & \mathbf{0} & \mathbf{0} \\
			0 & 0 & -\lambda &\varepsilon \, \mu \,\varphi_+(\mathbf{p}) \\
			0 & 1 & \lambda & \omega(v_1,v_2)\,\omega(v_0,v_3)/\sqrt{\Pi}
		\end{array}\right] = \Phi_3 \, \circ \, \pi_3(\mathbf{p}). \vspace{-20pt}
\end{equation*}
\end{proof}

\subsection{Proof of \autoref{param5}}
Recall the cross-ratios $\alpha_j, \beta_j, \gamma_j: \clP_r^5 \to \bbC$ from \autoref{alpha,beta,gamma}. With the notation established in \eqref{def_gamma}--\eqref{def_varphi}, we set $\psi_-,\psi_+:\bbC^6 \to \bbC$ to be a.e.
\begin{equation} \label{def_psi}
	\psi_{\eta}(\mathbf{a},\mathbf{b},\mathbf{c}) \coloneqq \eta \, \varphi_\eta(\mathbf{a}) \cdot \Gamma(\mathbf{b}) \ \ + \ \ a_2b_1c_1^{-1} \cdot \Gamma(\mathbf{c}), \quad \eta \in \{\pm 1\},
\end{equation}
for $\mathbf{a} = (a_1,a_2),\,\mathbf{b} = (b_1,b_2),\,\mathbf{c} = (c_1,c_2)$. Let also $\Gamma_\alpha, \, \Gamma_\beta, \, \Gamma_\gamma:\, \clP_r^5 \to \bbC$ be the functions defined a.e. as
\begin{equation*}
	\Gamma_\alpha \coloneqq \Gamma\,\circ\,(\alpha_1,\alpha_2), \quad \Gamma_\beta \coloneqq \Gamma\,\circ\,(\beta_1,\beta_2), \quad \Gamma_\gamma \coloneqq \Gamma\,\circ\,(\gamma_1,\gamma_2).
\end{equation*}
In this subsection, we will consider $\Delta, \, \Delta^{1/2},\,\varphi_\eta,\,\psi_\eta$ a.e. defined functions $\clP_r^{5} \to \bbC$ of 5-tuples after precomposing the expressions from \eqref{def_delta}--\eqref{def_varphi} with the function $(\alpha_1,\alpha_2)$, and the expression from \eqref{def_psi} with $(\alpha_1,\alpha_2,\beta_1,\beta_2,\gamma_1,\gamma_2)$. 

The proof of \autoref{param5} builds up on the computations of the previous subsection. The next lemma establishes the result of pairing the vectors $e'$, $f'$ obtained in \autoref{key lemma} from a tuple $[v_0,v_1,v_2,v_3] \in \clP_r^{\{\!\{4\}\!\}}$ with a fifth vector $v_4$. Its proof is a direct computation that we shall omit, relying on the explicit expressions of $e'$ and $f'$ given in  \eqref{ef_-1} and \eqref{ef_+1}.

\begin{lem} \label{continuation key lemma}
Let $[v_0,v_1,v_2,v_3] \in \clP_r^{\{\!\{4\}\!\}}$, let $\widehat\cdot$ be the complement map onto the hyperbolic plane $\langle v_0,v_1 \rangle$ spanned by the first two vectors, and let $\lambda,\mu,e',f'$ be as in \autoref{key lemma}. If an isotropic vector $v_4 \in V_r \smallsetminus \{0\}$ is such that $\mathbf{p} \coloneqq [v_0,v_1,v_2,v_3,v_4]$ lies in $\clP_r^{(5)}$, then 
\begin{equation*} 
	\omega(e',v_4) = \varepsilon \frac{\omega(v_2,v_4)}{\lambda} \cdot \frac{\psi_-(\mathbf{p})}{\Delta^{1/2}(\mathbf{p})} \qand
	\omega(f',v_4) = \frac{\omega(v_2,v_4)}{\lambda} \cdot \frac{\psi_+(\mathbf{p})}{\Delta^{1/2}(\mathbf{p})}. 
\end{equation*} 
\end{lem}

We give our definition of genericity for 5-tuples at this point, and then prove \autoref{param5}. 

\begin{defn} \label{def_curlyP5}
We let $\clP_r^{\{\!\{5\}\!\}}$ be the set of 5-tuples $(p_0,p_1,p_2,p_3,p_4)  \in \clP_r^{\{5\}}$ such that every 4-subtuple is in $\clP_r^{\{\!\{4\}\!\}}$. In virtue of \autoref{lem:alpha,beta,gamma}, such a tuple is in $\clP_r^{\{\!\{5\}\!\}}$ if and only if\vspace{-2pt}
\[
	(p_0,p_1,p_i,p_j) \in \clP_r^{\{\!\{4\}\!\}} \quad \mbox{for } i,j \in \{2,3,4\} \mbox{ with } i < j.\vspace{-1pt}
\]
One verifies that $\clP_5^{\{\!\{5\}\!\}}$ is a $G_r$-invariant, co-null subset of $\clP_r^{5}$. 
\end{defn}


\begin{proof}[Proof of \autoref{param5}]
For every $r \geq r_1 + 1$, we must  define a co-null subset $\Omega_4 \subset \bbC^5$ and a Borel map
$
	\Phi_4: \Omega_4 \to \clP_r^{5}
$
such that:
\begin{enumerate}[label=(\roman*),leftmargin=25pt]
	\item $\pi_4(\clP_r^{\{\!\{5\}\!\}}) \subset \Omega_4$.
	\item $\Phi_4(\Omega_4) \subset \clP_r^{\{\!\{5\}\!\}}$ and $\pi_4 \, \circ \, \Phi_4 = \id$.
	\item For all $\mathbf{p} \in \clP_r^{\{\!\{5\}\!\}}$, there exists $g \in G_r$ such that $\Phi_4 \, \circ \, \pi_4(\mathbf{p}) = g\, \mathbf{p}$. 
\end{enumerate}

We leave out the case $(\varepsilon,d) = (+1,1)$, $r=2$ for individual consideration, and assume that \vspace{-3pt}
\[
	r \geq \begin{cases}
			3	& \mbox{ if } \varepsilon = +1, \\[-4pt]
			2	& \mbox{ if } \varepsilon = -1,\vspace{-3pt}
		  \end{cases}
\]
Similarly to the definition of $\Omega_3$ in \eqref{omega3}, we let
\begin{align}
	\Omega_4 & \coloneqq \left\{(a_1,a_2,b_1,b_2,c_1) \in (\bbC^\times)^5 \mid	\Delta(a_1,a_2), \,\Delta(b_1,b_2), \, \Delta(c_1,c_2) \neq 0\right\}, \label{omega4}
\end{align}
where $c_2 \coloneqq \varepsilon (a_1b_2c_1)/(a_2 b_1)$. Being an intersection of finitely many non-empty Zariski-open subsets of $\bbC^5$, the set $\Omega_4$ is itself non-empty and Zariski-open, in particular co-null. \autoref{def_curlyP5} and \autoref{delta_nondeg} imply that 
$\Delta(\alpha_1(\mathbf{p}),\alpha_2(\mathbf{p})), \, \Delta(\beta_1(\mathbf{p}),\beta_2(\mathbf{p})), \, \Delta(\gamma_1(\mathbf{p}),\gamma_2(\mathbf{p})) \neq 0$ for every $\mathbf{p} \in \clP_r^{\{\!\{5\}\!\}}$, establishing (i). 

We define the functions $\tilde\phi_4 \colon \Omega_4 \to \langle e_r,e_{r-1},f_{r-1},f_r \rangle$ and $\phi_4:\Omega_4 \to V_r$ as
\begin{align}
	\tilde\phi_4(\mathbf{a},\mathbf{b},c_1) &\coloneqq b_1 \ e_r \ + \ \frac{\psi_{-}(\mathbf{a},\mathbf{b},\mathbf{c})}{\Delta^{1/2}(\mathbf{a})} \ e_{r-1} \ + \ \varepsilon \,\frac{\psi_{+}(\mathbf{a},\mathbf{b},\mathbf{c})}{\Delta^{1/2}(\mathbf{a})} \ f_{r-1} \ + \ \varepsilon \, b_2 \ f_r, \nonumber \\[4pt]
	\phi_4(\mathbf{a},\mathbf{b},c_1) &\coloneqq \tilde\phi_4(\mathbf{a},\mathbf{b},c_1) \ + \left\{e_{r-2} - \left(\frac{q\big(\tilde\phi_4(\mathbf{a},\mathbf{b},c_1)\big)}2\right) \cdot f_{r-2}\right\}, \label{phi4_ef}	
\end{align}
with $\mathbf{a}$, $\mathbf{b}$, $\mathbf{c}$ as in \eqref{def_psi} and  $c_2 \coloneqq \varepsilon (a_1b_2c_1)/(a_2 b_1)$. Observe that if $\varepsilon= +1$, the vector $\tilde\phi_4(\mathbf{a},\mathbf{b},c_1)$ is not necessarily isotropic. In virtue of the assumption that $r \geq 3$ in that setting, we are able to incorporate in $\phi_4(\mathbf{a},\mathbf{b},c_1)$ a correction that makes it isotropic for any $(\mathbf{a},\mathbf{b},c_1)$.

We define $\Phi_4\colon\Omega_4 \to \clP_r^5$,
\begin{equation*} \label{bigphi4}
	 \Phi_4(\mathbf{a},\mathbf{b},c_1) \coloneqq \left[e_r,\,f_r,\,\phi_2,\ \phi_3(\mathbf{a}),\ \phi_4(\mathbf{a},\mathbf{b},c_1)\right]. 
\end{equation*}
Note that the first four lines in $\Phi_4(\mathbf{a},\mathbf{b},c_1)$ correspond to the tuple $\Phi_3(\mathbf{a})$ introduced in \eqref{bigphi3}. 
Abbreviating $\phi_3 \equiv \phi_3(\mathbf{a})$ and $\phi_4 \equiv \phi_4(\mathbf{a},\mathbf{b},c_1)$, we obtain, in addition to \eqref{omega_pairing}, the identities
\begin{equation} \label{omega_pairing2}
	\omega(e_r,\phi_4) = \varepsilon b_2, \quad \omega(f_r,\phi_4) = \varepsilon b_1, \quad
	\omega(\phi_2,\phi_4) = \varepsilon, \quad \omega(\phi_3,\phi_4) = a_2b_1c_1^{-1}.
\end{equation} 
In particular, this proves that $\Phi_4$ ranges in $\clP_r^{(5)}$.  
%
From \eqref{omega_pairing} and \eqref{omega_pairing2}, we derive 
\begin{align*}
 \alpha_1\big(\Phi_4(\mathbf{a},\mathbf{b},c_1)\big) &= a_1, \qquad  
  \beta_1\big(\Phi_4(\mathbf{a},\mathbf{b},c_1)\big) = b_1, \qquad
  \gamma_1\big(\Phi_4(\mathbf{a},\mathbf{b},c_1)\big) = c_1, \\
  \alpha_2\big(\Phi_4(\mathbf{a},\mathbf{b},c_1)\big) &= a_2, 
 \qquad
 \beta_2\big(\Phi_4(\mathbf{a},\mathbf{b},c_1)\big) = b_2.
\end{align*}
This implies that $\pi_4 \, \circ \, \Phi_4 = \id$ and that every 4-subtuple of $\Phi_4(\mathbf{a},\mathbf{b},c_1)$ lies in $\clP_r^{\{\!\{4\}\!\}}$. In order to conclude that $\Phi_4(\mathbf{a},\mathbf{b},c_1) \in \clP_r^{\{\!\{5\}\!\}}$ holds, it suffices to observe that if $\dim(V_r) = 2r+d \geq 5$, then necessarily $r > 2$, so that the correction term in \eqref{phi4_ef} is non-zero, and hence, the lines in $\Phi_4(\mathbf{a},\mathbf{b},c_1)$ are linearly independent. This proves point (ii).

For (iii), fix an arbitrary $\mathbf{p} = [v_0,v_1,v_2,v_3,v_4] \in \clP_r^{\{\!\{5\}\!\}}$ and let $\{e'_r,e'_{r-1},f'_{r-1},f_r'\}$ be the adapted basis of the subspace $\langle v_0,v_1,v_2,v_3 \rangle$, 
\begin{equation*} 
	e'_r \coloneqq \frac{\lambda \, v_1}{\omega(v_1,v_2)}, \quad e'_{r-1} \coloneqq e', \quad f'_{r-1} \coloneqq f', \quad f'_r \coloneqq \frac{\omega(v_1,v_2)\,v_0}{\sqrt{\Pi}},
\end{equation*}
defined as in \eqref{4-basis}, where $\Pi,\lambda,\mu,e',f'$ are functions of $(v_0,v_1,v_2,v_3)$ as in \autoref{key lemma}. Let also $g = g(v_0,v_1,v_2,v_3) \in G_r$ be the element from \eqref{defn_g}, so that 
\[
	g\cdot [v_0,v_1,v_2,v_3] = \Phi_3(\pi_3(\mathbf{p})).
\]
Let $x_0 \in \langle v_0,v_1,v_2,v_3\rangle^\perp$ be the vector such that 
\[
	g \cdot v_4 = x_0 + \omega(e'_r,v_4)\, e_r + \omega(e'_{r-1},v_4)\, e_{r-1} + \omega(f'_{r-1},v_4)\, f_{r-1} + \omega(f'_r,v_4)\, f_r. \\
\]
Since the span of the lines in $g [v_0,v_1,v_2,v_3]$ equals $\langle e_r,e_{r-1},f_{r-1},f_r \rangle$, we have that $x_0 \neq 0$ whenever $r > 2$. In coordinates with respect to $\scB_r$, we have
\begin{align*}
	g \cdot [v_4] &= \left[ \begin{array}{c}
			\omega(e'_r,v_4)\\[5pt]
			\omega(e'_{r-1},v_4) \\[4pt]
			x_0 \\[4pt]
			\omega(f'_{r-1},v_4)\\[4pt]
			\omega(f'_r,v_4)
		\end{array}\right] 
		= \left[ \begin{array}{c}
			\lambda \,  \omega(v_1,v_4)/\omega(v_1,v_2)\\[5pt]
			\varepsilon \frac{\omega(v_2,v_4)}{\lambda} \cdot \frac{\psi_-(\mathbf{p})}{\Delta^{1/2}(\mathbf{p})} \\[4pt]
			x_0 \\[4pt]
			\frac{\omega(v_2,v_4)}{\lambda} \cdot \frac{\psi_+(\mathbf{p})}{\Delta^{1/2}(\mathbf{p})}\\[4pt]
			\omega(v_1,v_2) \, \omega(v_0,v_4)/\sqrt\Pi
		\end{array}\right]
		= \left[ \begin{array}{c}
			\beta_1(\mathbf{p})\\[5pt]
			\frac{\psi_-(\mathbf{p})}{\Delta^{1/2}(\mathbf{p})}  \\[5pt]
			x_0' \\[5pt]
			\varepsilon \, \frac{\psi_+(\mathbf{p})}{\Delta^{1/2}(\mathbf{p})} \\[5pt]
			\varepsilon \, \beta_2(\mathbf{p}) 
		\end{array}\right] = \left[\tilde\phi_4 \,(\pi_4(\mathbf{p}))+x_0'\right]
\end{align*}
where $x'_0 \coloneqq \varepsilon \lambda/\omega(v_2,v_4) \cdot x_0 \neq 0$. Indeed, the last equality holds by \autoref{lem:alpha,beta,gamma}, which implies that $\gamma_2(\mathbf{p}) = \varepsilon(\alpha_1\alpha_2^{-1}\beta_1^{-1}\beta_2\gamma_1)(\mathbf{p})$. 

Note that there exists an element $g' \in G_{r-2} < G_r$ such that $g'x_0 = e_{r-2} + (q(x_0')/2) \, f_{r-2}$, and that
\[
q(x'_0) = -q\big(\tilde\phi_4 \,(\pi_4(\mathbf{p})\big).
\]
Therefore, $(g'g)\cdot \mathbf{p} = \Phi_4 \, \circ \, \pi_4(\mathbf{p})$, completing the proof of (iii).

The proof for the case $(\varepsilon,d) = (+1,1)$, $r = 2$ is analogous, but requires a few modifications. We must replace everywhere in the argument above the set $\Omega_4$ as in \eqref{omega4} by its co-null subset $\Omega_4' \subset \Omega_4$ defined below, and the map $\phi_4$ from \eqref{phi4_ef} by $\phi'_4: \Omega_4' \to V_r$:
\begin{align*}
	&\Omega_4' \coloneqq \left\{(\mathbf{a},\mathbf{b},c_1) \in \Omega_4 \left| \ q\big(\tilde\phi_4(\mathbf{a},\mathbf{b},c_1) \big) \neq 0 \right.\right\}, \\
	&\phi_4'(\mathbf{a},\mathbf{b},c_1) \coloneqq \tilde\phi_4(\mathbf{a},\mathbf{b},c_1) \ + \sqrt{-q\big(\tilde\phi_4(\mathbf{a},\mathbf{b},c_1)\big)} \cdot h.\qedhere
\end{align*}
\end{proof}

\begin{rem} 
The isomorphism $G_1\lmod \clP_1^5 \cong \bbC^2$ of Lebesgue spaces holds for the parameters $(\varepsilon,d) \in \{(+1,1),(-1,0)\}$. Indeed, in the classical setting (see \autoref{rem:classical}), any 5-tuple $\mathbf{p} \in \clP_1^{\{\!\{5\}\!\}}$ is in the $G_1$-orbit of the tuple $(\infty,0,1,a,b) \in \widehat\bbC^5$ with $a=\alpha_0(\mathbf{p})$ and $b=\beta_0(\mathbf{p})$. 

In the other case that is not covered by \autoref{param5}, namely $(\varepsilon,d)=(+1,0)$ and $r=2$, low dimensionality forces the extra relation $q(\tilde\phi_4(\pi_4(\mathbf{p})))=0$, which causes that $G_2\lmod \clP_2^5 \cong \bbC^4$. 
\end{rem}



\begin{thebibliography}{123}
\bibitem{Bloch}
	S.J. Bloch.
	\emph{Higher regulators, algebraic K-theory, and zeta functions of elliptic curves.}
	CRM Monograph Series, 11. 
	American Mathematical Society, Providence, RI, 2000.
	
\bibitem{BW} 
	A. Borel, N. Wallach.
	\emph{Continuous cohomology, discrete subgroups, and representations of reductive groups}. 
	Second edition. Mathematical Surveys and Monographs, 67. 
	American Mathematical Society, Providence, RI, 2000. 
	
\bibitem{BBI}
	M. Bucher, M. Burger, A. Iozzi. 
	\emph{The bounded Borel class and 3-manifold groups}.
	Duke Math. J. 167, 
	no. 17 (2018), 3129--3169.
	
\bibitem{BIW}
	M. Burger, A. Iozzi, A. Wienhard.
	\emph{Surface group representations with maximal Toledo invariant}.
	Ann. Math. 172 (2010), 
	no. 1, 517--566.

\bibitem{Burger-Monod1}
	M. Burger, N. Monod.
	\emph{Bounded cohomology of lattices in higher rank Lie groups.} 
	J. Eur. Math. Soc. 1 (1999), 
	no. 2, 199--235. 
	
\bibitem{Burger-Monod2}
	M. Burger, N. Monod.
	\emph{Continuous bounded cohomology and applications to rigidity theory}. 
	Geom. Funct. Anal. 12 (2002), 
	no. 2, 219--280. 

\bibitem{Burger-Monod3}
	M. Burger, N. Monod.
	\emph{On and around the bounded cohomology of $\SL_2$}.
	Rigidity in dynamics and geometry, 19--37, Springer, Berlin, 2002. 
	
\bibitem{Campagnolo-etal}
	C. Campagnolo, F. Fournier-Facio, N. Heuer, M. Moraschini.
	\emph{Bounded Cohomology and Simplicial Volume}.
	London Mathematical Society Lecture Note Series,
	Cambridge University Press, Cambridge, 2022.
		
\bibitem{Dupont} 
	J. L. Dupont.
	\emph{Bounds for characteristic numbers of flat bundles}. 
	Algebraic topology, Aarhus 1978 (Proc. Sympos., Univ. Aarhus, Aarhus, 1978), pp. 109-119, 
	Lecture Notes in Math., 763, Springer, Berlin, 1979.

\bibitem{DM-Thesis}
	C. De la Cruz Mengual, 
	\emph{On Bounded-Cohomological Stability for Classical Groups.} 
	ETH Zurich (2019), Zurich.

	
\bibitem{DMH0}
	C. De la Cruz Mengual, T. Hartnick.
	\emph{A Quillen stability criterion for bounded cohomology.} arXiv preprint (2023), \href{https://arxiv.org/abs/2307.12808}{arXiv:2307.12808}.
		
\bibitem{DM+Hartnick}
	C. De la Cruz Mengual, T. Hartnick.
	\emph{Stabilization of Bounded Cohomology for Classical Groups}.
	arXiv preprint (2022), \href{https://arxiv.org/abs/2201.03879}{arXiv:2201.03879}.

\bibitem{Farre}
	J. Farre.
	\emph{Relations in bounded cohomology.} 
	J. Topol. 13 (2020), 
	no. 3, 1084--1118.


\bibitem{Frigerio}
	R. Frigerio.
	\emph{Bounded cohomology of discrete groups}.
	Mathematical Surveys and Monographs, 227. 
	American Mathematical Society, 
	Providence, RI, 2017. 

\bibitem{Galatius}
	S. Galatius, A. Kupers, O. Randal-Williams.
	$\mathrm{E}_2$-cells and mapping class groups. 
	Publ. Math., Inst. Hautes \'Etud. Sci. 130 (2019), 
	1--61.

\bibitem{Goldman}
	W. M. Goldman.
	\emph{Complex hyperbolic geometry.}
	Oxford Mathematical Monographs. Oxford Science Publications. 
	Clarendon Press. xx, 316 p.m Oxford, 1999.
	
\bibitem{Goncharov}
	 A.B. Goncharov.
	\emph{Explicit construction of characteristic classes}. 
	I. M. Gel'fand Seminar, 169--210, 
	Adv. Soviet Math., 16 (1993), Part 1, Amer. Math. Soc., Providence, RI.
	
\bibitem{GHV} 
	W. Greub, S. Halperin, R. Vanstone.
	\emph{Connections, curvature and cohomology, Volume III: Cohomology of principal bundles and homogeneous spaces}. 
	Pure and Applied Mathematics, Vol. 47-III. 
	Academic Press [Harcourt Brace Jovanovich, Publishers], New York--London, 1976.


\bibitem{Gromov}
	M. Gromov.
	\emph{Volume and bounded cohomology}.
	Inst. Hautes \'Etudes Sci. Publ. Math. 56 (1982), 
	5--99. 

\bibitem{HartOtt}
	T. Hartnick, A. Ott.
	\emph{Bounded cohomology via partial differential equations, I.}
	Geom. Topol. 19 (2015), 
	no. 6, 3603--3643. 
	
\bibitem{HartOtt2}
	T. Hartnick, A. Ott.
	\emph{Surjectivity of the Comparison Map in Bounded Cohomology for Hermitian Lie Groups}.
	Int. Math. Res. Not. IMRN 2012 (2012), 
	no. 9, 2068--2093.
	

\bibitem{KasSr}
	T. Kastenholz, R. J. Sroka.
	\emph{Simplicial bounded cohomology and stability}.
	arXiv preprint (2023), \href{https://arxiv.org/abs/2309.05024}{arXiv:2309.05024}.

\bibitem{KR}
	A. Kor\'anyi, H. M. Reimann.
	\emph{The complex cross ratio on the Heisenberg group.} 
	Enseign. Math., II. S\'er. (1987),
	no. 33, 291--300.
	
\bibitem{Kos}
	B. Kostant.
	\emph{The principal three-dimensional subgroup and the Betti numbers of a complex simple Lie group.}
	Am. J. Math. 81 (1959), 
	973--1032.
	
\bibitem{LS}
	J-F. Lafont, B. Schmidt.
	\emph{Simplicial volume of closed locally symmetric spaces of noncompact
type.} 
	Acta Math. 197 (2006), 
	no. 1, 129--143.
	
\bibitem{Toda-Mimura}
	M. Mimura, H. Toda.
	\emph{Topology of Lie groups. I, II}. 
	Translated from the 1978 Japanese edition by the authors. Translations of Mathematical Monographs, 91. 
	American Mathematical Society, Providence, RI, 1991.
	
\bibitem{Monod-Survey}
	N. Monod.
	\emph{An invitation to bounded cohomology}. 
	International Congress of Mathematicians. Vol. II, 1183--1211, 
	Eur. Math. Soc., 
	Z\"urich, 2006. 

\bibitem{Monod-Book} 
	N. Monod.
	\emph{Continuous bounded cohomology of locally compact groups}. 
	Lecture Notes in Mathematics, 1758. 
	Springer-Verlag, 
	Berlin, 2001. 
	
\bibitem{Monod-Lifting}
	N. Monod. 
	\emph{Equivariant measurable liftings}.
	Fundam. Math. 230 (2015), 
	no. 2, 149--165.
	

\bibitem{Monod-Stab} 
	N. Monod.
	\emph{Stabilization for $\SL_n$ in bounded cohomology}. 
	Contemp. Math. 347 (2004), 
	191--202. 
	
\bibitem{Oni}
	A. L. Onishchik.
 	\emph{Topology of Transitive Transformation Groups.} 
	Johann Ambrosius Barth Verlag GmbH, Leipzig (1994). 
	
\bibitem{Panyushev}
	D. I. Panyushev.
	\emph{On the Dynkin index of a principal $\mathfrak{sl}_2$-subalgebra}.
	 Adv. Math. 221 (2009), 
	 no. 4, 1115--1121.


	
\bibitem{Stas}
	J. D. Stasheff.
	\emph{Continuous cohomology of groups and classifying spaces}.
	Bull. Amer. Math. Soc. 84 (1978),
	no. 4, 513--530.
	
\bibitem{Tao}
	T. Tao.
	\emph{Exceptional isogenies between the classical groups}.
	Blog entry in "What's new" (11 March, 2011) \href{https://terrytao.wordpress.com/2011/03/11/exceptional-isogenies-between-the-classical-lie-groups/}{https://terrytao.wordpress.com/2011/03/11/exceptional-isogenies-between-the-classical-lie-groups/}.

\bibitem{Taylor}
	D.E. Taylor. 
	\emph{The geometry of the classical groups.} 
	Sigma Series in Pure Mathematics, 9. Heldermann Verlag, Berlin, 1992. 
	
\bibitem{vdWaerden}
	B. L. van der Waerden.
	\emph{Gruppen von Linearen Transformationen}.
	Springer-Verlag, Berlin, 1935.

\bibitem{Zagier}
	 D. Zagier.
	 \emph{The dilogarithm function}. 
	 Frontiers in number theory, physics, and geometry. II, 3--65, Springer, Berlin, 2007.
	 
\bibitem{Zimmer}
	R. J. Zimmer.
	\emph{Ergodic theory and semisimple groups}.
	Monographs in Mathematics, Vol. 81. Birkh\"auser, Boston-Basel-Stuttgart, 1984.
	
\end{thebibliography}
\end{document}